\documentclass[11pt, twoside, leqno]{article}

\usepackage{amsmath}
\usepackage{amssymb}
\usepackage{titletoc}
\usepackage{mathrsfs}
\usepackage{amsthm}
\usepackage{indentfirst}
\usepackage{color}
\usepackage{txfonts}

\allowdisplaybreaks

\def\bint{{\ifinner\rlap{\bf\kern.30em--}
\int\else\rlap{\bf\kern.35em--}\int\fi}\ignorespaces}

\def\sbint{{\ifinner\rlap{\bf\kern.32em--}
\hspace{0.078cm}\int\else\rlap{\bf\kern.45em--}\int\fi}\ignorespaces}

\def\rr{{\mathbb R}}
\def\rn{{{\rr}^n}}

\def\nn{{\mathbb N}}

\def\cc{{\mathbb C}}

\def\cf{{\mathcal F}}

\def\cm{{\mathcal M}}

\def\cb{{\mathcal B}}

\def\co{{\mathcal O}}

\def\fz{\infty}
\def\az{\alpha}
\def\supp{{\mathop\mathrm{\,supp\,}}}

\def\loc{{\mathop\mathrm{\,loc\,}}}

\def\dz{\delta}

\def\ez{\epsilon}

\def\wz{\widetilde}

\def\ls{\lesssim}
\def\gs{\gtrsim}

\def\bmo{{{\rm BMO}(\rn)}}

\def\cmoc{{{\rm CMO}(\cc)}}
\def\bmoc{{{\rm BMO}(\cc)}}
\def\linc{{{L^\fz}(\cc)}}

\def\r{\right}
\def\lf{\left}

\def\beeqn{\begin{equation}}
\def\eneqn{\end{equation}}
\def\beeqns{\begin{equation*}}
\def\eneqns{\end{equation*}}
\def\beeqa{\begin{eqnarray}}
\def\eneqa{\begin{eqnarray}}
\def\beeqas{\begin{eqnarray*}}
\def\eneqas{\begin{eqnarray*}}
\def\besp{\begin{split}}
\def\ensp{\begin{split}}
\def\noz{\nonumber}

\newtheorem{thm}{Theorem}[section]
\newtheorem{lem}[thm]{Lemma}%[section]
\theoremstyle{definition}

 \newtheorem{defn}[thm]{Definition}%[section]
\numberwithin{equation}{section}

\begin{document}

\title{\Large\bf Buerling-Ahlfors Commutators on Weighted Morrey Spaces
and Applications to Beltrami Equations
\footnotetext{\hspace{-0.35cm} 2010 {\it
Mathematics Subject Classification}: Primary 42B20;
Secondary 46E35. \endgraf
{\it Key words and phrases}. Beurling-Ahlfors transform, compactness, commutator,
weighted Morrey space, Beltrami equation.\endgraf
 \endgraf
Dachun Yang is supported by the National Natural Science Foundation of China
(Grant Nos. 11571039, 11726621 and 11761131002).
Dongyong Yang is supported by the National Natural Science Foundation of China
(Grant No. 11571289) and Fundamental Research Funds for Central Universities of China
(Grant No. 20720170005).}}
\date{ }
\author{Jin Tao, Dachun Yang and Dongyong Yang\footnote{Corresponding author/April 7, 2018/newest version.}}
\maketitle

\vspace{-0.8cm}

\begin{center}
\begin{minipage}{13cm}
{\noindent  {\bf Abstract:}\ Let $p\in(1, \infty)$, $\kappa\in(0, 1)$ and
$w\in A_p(\cc).$ In this article, the authors obtain a boundedness
(resp., compactness) characterization of the Buerling-Ahlfors commutator
$[\mathcal B, b]$ on the weighted Morrey space $L_w^{p,\,\kappa}(\mathbb C)$
via $\mathop\mathrm{BMO}(\cc)$ [resp., $\mathop\mathrm{CMO}(\cc)$],
where $\mathcal B$ denotes the Buerling-Ahlfors transform and
$b\in \mathop\mathrm{BMO}(\cc)$ [resp., $\mathop\mathrm{CMO}(\cc)$].
Moreover, an application to the Beltrami equation is also given.}
\end{minipage}
\end{center}

\vspace{0.0cm}

\section{Introduction and statement of main results\label{s1}}

Let $T$ be a {Calder\'on-Zygmund operator} on $\rn$ and $b\in L^1_{\rm loc}(\rn)$.
The commutator $[b, T]$ is defined by setting
$$  [b,T]f(x):=b(x)T(f)(x)-T(bf)(x)$$
for suitable function $f$ with compact support and $x\notin \supp (f)$.
It is well known that the boundedness (resp., compactness) of
Calder\'on-Zygmund commutators on function spaces can be characterized by functions in
${\rm BMO}(\rn)$ [resp., ${\rm CMO}(\mathbb R^n)$] and plays an important role
in harmonic analysis, complex analysis, PDEs and other fields in mathematics.
Here and hereafter, the \emph{space} ${\rm CMO}(\mathbb R^n)$ is defined to be the
${\rm BMO}(\mathbb R^n)$-closure of $C^\fz_c(\rn)$, the set of all infinitely
differentiable functions on $\rn$ with compact supports.

In particular, to extend the classical $H^p$ spaces to higher dimension,
Coifman et al.\,\cite{CoifmanRochbergWeiss76AnnMath}
proved that, for any $b\in {\rm BMO}(\mathbb R^n)$, the commutator $[b,T]$
of a Calder\'on-Zygmund operator $T$ with smooth kernel is bounded on
$L^p(\mathbb R^n)$ for any $p\in(1, \infty)$;
they also proved that, if $[b, R_j]$ is bounded on $L^p(\mathbb R^n)$
for every Riesz transform $R_ j$, $j\in\{1,\ldots, n\}$ and some $p\in(1,\infty)$,
then $b\in {\rm BMO}(\mathbb R^n)$.
Later, Uchiyama \cite{Uchiyama78TohokuMathJ} further showed that
$b\in{\rm BMO}(\rn)$ if the commutator $[b,T]$ of a Calder\'on-Zygmund operator
$T$ with smooth kernel is bounded on $L^p(\rn)$ for some $p\in(1, \infty)$;
he also showed that $[b, T]$ is compact on $L^p(\mathbb R^n)$ for any $p\in(1, \infty)$
if and only if $b\in{\rm CMO}(\mathbb R^n)$.  These equivalent characterizations
of the boundedness and the compactness of commutators were further extended to the
Morrey space $L^{p,\,\kappa}(\rn)$ by Di Fazio and Ragusa \cite{DiFazioRagusa91BUMIA}
and Chen et al.\,\cite{CDW12CJM}, respectively, where $\kappa\in(0,1)$ and $p\in(1,\infty)$.

On the other hand,  Coifman et al.\,\cite{CoifmanLMS93JMPA} applied the bounedness of
Calder\'on-Zygmund commutators to study Navier-Stokes equations.
Let $\cb$ be the Beurling-Ahlfors transform on $\cc$ defined by
the following principal value:
$$\cb f(z):=\,\mathrm{p.\,v.\,}-\frac1\pi\int_{\mathbb C}\frac{f(u)}{(z-u)^2}\,du.$$
For brevity, we denote the area element $dx\,dy$ in
$\mathbb{R}^2$ (or, equivalently, the differential form $du\wedge d\bar{u}$ in $\cc$)
by $du$ as above. Then $\cb$ is a Calder\'on-Zygmund operator on $\cc$. More precisely,
let $K_{\cb}(z,u):=-\frac1\pi\frac1{(z-u)^2}$ be the kernel of $\cb$.
Then there exists a positive constant $C$ such that
\begin{itemize}
\item [{\rm(i)}] for any $z,\,u\in\cc$ with $z\not=u$,
  \begin{align}\label{cz kernel condition-1}
  |K_{\cb}(z,u)|\le C \frac1{|z-u|^2};
  \end{align}
\item[{\rm(ii)}] for any $z,\,u,\,u_0\in \cc$ with $|u_0-u|\le|u_0-z|/2$,
  \begin{align}\label{cz kernel condition-2}
   &|K_{\cb}(u_0, z)-K_{\cb}(u, z)|+|K_{\cb}(z, u_0)-K_{\cb}(z, u)|
   \le C \frac{|u_0-u|}{|u_0-z|^{3}}.
  \end{align}
\end{itemize}
At the frequency side, $\cb$ corresponds to the Fourier multiplier $m(\xi):=\frac{\overline{\xi}}{\xi}$ and $\cb$ is an isometry in $L^2(\mathbb C)$.
For more work on Beurling-Ahlfors transforms, we refer the reader to \cite{AstalaIwaniecSaksman01DMJ,AstalaIwaniecMartin09,MateuOrobitgVerdera09JMPA, GutlyanskiiRyazanovSrebroYakubov12,BojarskiGutlyanskiiMartioRyazanov13,ClopCruz13AASFM}.

In \cite{Iwaniec92}, Iwaniec used the $L^p(\rn)$-compactness theorem of Calder\'on-Zygmund
commutators of Uchiyama \cite{Uchiyama78TohokuMathJ} to derive the
$L^p(\mathbb C)$-invertibility of the operator $Id-b\cb$, where $p\in(1,\fz)$,
$Id$ is the identity operator and the coefficient
$b\in L^\fz(\mathbb C)\cap {\rm CMO}(\mathbb C)$
has compact support, and applied it to study linear complex Beltrami equations and
the $L^p(\cc)$-theory of quasiregular mappings.
These results were further extended to the weighted Lebesgue space $L^p_w(\mathbb C)$
with $p\in(1,\fz)$ and $w\in A_p(\mathbb C)$ by
Clop and Cruz \cite{ClopCruz13AASFM}, where they also
obtained a priori  estimate in $L^p_w(\mathbb C)$ for the generalized Beltrami
equation and the regularity for the Jacobian of certain quasiconformal mappings. 
For more results on the boundedness and the compactness of  Calder\'on-Zygmund
commutators  on function spaces and their applications, please see \cite{Janson78ArkMat,CoifmanLMS93JMPA,KrantzLi01JMAAb,KomoriMizuhara06MathNachr,
MaoSunWu16ActaMathSin,LernerOmbrosiRivera17arxiv}
and references therein.

Let $p\in(1,\infty)$, $\kappa\in(0,1)$ and $w\in A_p(\cc)$. In this article,
we consider the boundedness and the compactness characterizations of
the Beurling-Ahlfors transform commutator $[b,\mathcal B]$
on the weighted Morrey spaces $L_w^{p,\,\kappa}(\cc)$.
As an application, we apply the $L_w^{p,\,\kappa}(\cc)$-compactness
of $[b,\mathcal B]$ to study the Beltrami equation.
To this end, we first recall some necessary notation and notions.
In what follows, for any $p\in[1,\fz)$, we use the \emph{symbol $L^p_{\rm loc}(\cc)$}
to denote the set of all locally integrable functions on $\cc$.

\begin{defn}\label{d-Ap}\rm
Let $p\in(1, \infty)$. A non-negative function $w\in L^1_{\rm loc}(\cc)$ is called a
\emph{Muchenhoupt $A_p(\cc)$ weight}, denoted by $w\in A_p(\cc)$, if
\begin{align*}
[w]_{A_p(\cc)}:=\sup_{Q} \langle w\rangle_Q\langle w^{1-p'}\rangle_Q^{p-1}<\infty,
\end{align*}
where the supremum is taken over all squares $Q$ in $\cc$, $w (Q):=\int_{Q}w (z)\,dz$ and
$\langle w\rangle_Q:=\frac1{|Q|}w(Q)$.
\end{defn}

Throughout this article, for any $z\in\cc$ and $r\in(0,\infty)$,
let $Q(z,r)$ be the square in $\cc$ with center $z$ and side-length $2r$.
We recall the following notion of the weighted Morrey spaces from
\cite{KomoriShirai09MathNachr}.
\begin{defn}
Let $p\in(1,\infty)$, $\kappa\in(0,1)$ and $w\in A_p(\cc)$.
The \emph{weighted Morrey space} $L_w^{p,\,\kappa}(\cc)$ is defined by setting
\begin{equation*}
L_w^{p,\,\kappa}(\cc):=\lf\{f\in L^p_{\rm loc}(\cc):\,\,\|f\|_{L_w^{p,\,\kappa}(\cc)}<\infty\r\}
\end{equation*}
with
\begin{equation*}
\|f\|_{L_w^{p,\,\kappa}(\cc)}:=\sup_{r\in(0,\,\fz),\,z\in\mathbb C}
\lf\{\frac{1}{[w(Q(z,\,r))]^\kappa}\int_{Q(z,\,r)}|f(u)|^pw(u)\,du\r\}^{1/p}.
\end{equation*}
\end{defn}

In \cite{KomoriShirai09MathNachr}, Komori and Shirai obtained some results concerning
commutators on the weighted Morrey spaces $L^{p,\,\kappa}_w({\rn})$, where
$\kappa\in(0,1)$, $p\in(1, \infty)$ and $w$ is a Muckenhoupt $A_p$ weight on $\rn$;
they showed that  any Calder\'on-Zygmund operator $T$ and its commutator $[b,T]$
with $b\in\bmo$ are both bounded on  $L^{p,\,\kappa}_w({\rn})$.
Since the Beurling-Ahlfors transform $\cb$ is a Calder\'on-Zygmund operator,
we know that, for any $p\in(1,\infty)$, $\kappa\in(0,1)$ and $w\in A_p(\cc)$,
$\cb$ and the commutator $[b,\,\cb]$ with $b\in {\rm BMO}(\mathbb C)$ are both
bounded on $L_w^{p,\,\kappa}(\mathbb C)$. Then we have the following result.

\begin{thm}\label{t-Beurling com bounded}
Let $p\in(1,\infty)$, $\kappa\in(0,1)$, $w\in A_p(\cc)$ and $b\in L^1_{\loc}(\cc)$.
Then the Beurling-Ahlfors transform commutator $[b,\,\cb]$ has the following
boundedness characterization:
\begin{enumerate}
    \item[{\rm(i)}]If $b\in {\rm BMO}(\mathbb C)$, then $[b,\,\cb]$ is bounded on
    $L_w^{p,\,\kappa}(\mathbb C)$.
    \item[{\rm(ii)}]If $b$ is real-valued and $[b,\,\cb]$ is bounded on
    $L_w^{p,\,\kappa}(\mathbb C)$, then $b\in {\rm BMO}(\mathbb C)$.
\end{enumerate}
\end{thm}

Based on Theorem \ref{t-Beurling com bounded}, we further investigate the compactness
of the Buerling-Ahlfors transform commutator.
\begin{thm}\label{t-Beurling com compact}
Let $p\in(1,\infty)$, $\kappa\in(0,1)$, $w\in A_p(\cc)$ and $b\in {\rm BMO}(\mathbb C)$.
Then the Beurling-Ahlfors transform commutator $[b,\,\cb]$ has the following
compactness characterization:
\begin{enumerate}
    \item[{\rm(i)}]If $b\in {\rm CMO}(\mathbb C)$, then $[b,\,\cb]$ is compact on
    $L_w^{p,\,\kappa}(\mathbb C)$.
    \item[{\rm(ii)}]If $b$ is real-valued and $[b,\,\cb]$ is compact on
    $L_w^{p,\,\kappa}(\mathbb C)$, then $b\in {\rm CMO}(\mathbb C)$.
\end{enumerate}
\end{thm}

As an application of Theorem \ref{t-Beurling com compact}, we have the following result
on the Beltrami equation. In what follows, $\partial:=\frac{\partial}{\partial z}$,
$\bar\partial:=\frac{\partial}{\partial \bar z}$ and, for any $r\in(1,\fz)$, the
\emph{Lebesgue space} $L^r(\cc)$ is defined to be the set of all measurable functions
$f$ such that
$$\|f\|_{L^r(\cc)}:=\lf[\int_\cc |f(u)|^r\,du\r]^{1/r}<\fz.$$

\begin{thm}\label{t-Bretrami equ}
Let $p\in(1,\infty)$, $\kappa\in(0,1)$, $w\in A_p(\cc)$ and $b\in{\rm CMO}(\cc)$
such that $b$ has a compact support and  $\|b\|_{\linc}<1$. Then the equation
\begin{equation}\label{e-Bre equ}
\bar\partial f(z)-b(z)\partial f(z)=g(z)\quad w\mathrm{-}a.\,e.\ z\in\cc
\end{equation}
has, for any $g\in L_w^{p,\,\kappa}(\cc)\bigcap L^r(\cc)$ with some $r\in(1,\infty)$,
a solution $f$ with $|Df|:=|\partial f|+ |\overline{\partial}f|\in L_w^{p,\,\kappa}(\cc)$,
which is unique up to an additive constant.
Moreover, there exists a positive constant $C$, depending on $b$, $p$ and $\kappa$,
such that
\begin{equation}\label{e-Bre equ regu}
\lf\|\lf|Df\r|\r\|_{L_w^{p,\,\kappa}(\cc)}\le C\|g\|_{L_w^{p,\,\kappa}(\cc)}.
\end{equation}
\end{thm}

An outline of this article is in order.

In Section \ref{s2.5}, we give the proof of Theorem \ref{t-Beurling com bounded}.
In this section, we first obtain a simple but useful auxiliary lemma
(see Lemma \ref{product set of b} below), which is on the domination of
$|b(z)-\alpha_{\wz Q}(b)|$ for %the median value $\alpha_{\widetilde Q}(b)$ of
a given real-valued function $b\in L^1_{\mathrm{loc}}(\cc)$ by the difference
$|b(z)-b(u)|$ pointwise on subsets of $Q\times \widetilde Q$, where $Q$ and $\widetilde Q$
are squares and $\alpha_{\widetilde Q}(b)$ is the median value of $b$ over $\wz Q$ .
Compared to  \cite{LernerOmbrosiRivera17arxiv,TaoYangYang18arxiv},
our method adopted in the proof of Theorem \ref{t-Beurling com bounded} avoids
the use of the so-called local mean oscillation.

Section \ref{s3} is devoted to the proof of Theorem \ref{t-Beurling com compact} and
is divided into two subsections.
We give the proof of Theorem \ref{t-Beurling com compact}(i) in Subsection \ref{s3.1}.
Observe that, for any $p\in(1,\fz)$, $\kappa\in(0,1)$ and a general weight
$w\in A_p(\cc)$, $L_w^{p,\,\kappa}(\cc)$ is not invariant under translations.
Thus, in the proof of Theorem \ref{t-Beurling com compact}(i),
we use some ideas from \cite{KrantzLi01JMAAb,ClopCruz13AASFM} via first establishing the
boundedness of a maximal operator $\mathcal B_\ast$ of $\{\mathcal B_\eta\}_{\eta\in(0,\fz)}$,
a family of smooth truncated Beurling-Ahlfors transforms, on $L_w^{p,\,\kappa}(\cc)$.
Applying a version of the Fr\'echet-Kolmogorov theorem suitable for $L_w^{p,\,\kappa}(\cc)$,
and the $L_w^{p,\,\kappa}(\cc)$-boundedness of both $\mathcal B_\ast$ and
the Hardy-Littlewood maximal operator $\cm$,  we obtain the
$L_w^{p,\,\kappa}(\cc)$-compactness of the commutator $[b, \mathcal B_\eta]$ for
$b\in C^\infty_c(\cc)$. With a density argument involving the
$L_w^{p,\,\kappa}(\cc)$-boundedness of $[b, \mathcal B]$ and $\cm$, we 
further show the  $L_w^{p,\,\kappa}(\cc)$-compactness of $[b, \mathcal B]$ for any $b\in {\rm CMO}(\cc)$.

Subsection \ref{s3.2} is devoted to the proof of Theorem \ref{t-Beurling com compact}(ii).
As in the unweighted case (see, for example, \cite{Uchiyama78TohokuMathJ,TaoYangYang18arxiv}),
we first obtain a lemma for the upper and the lower bounds of integrals of
$[b, \mathcal B]f_j$ related to certain squares $Q_j$,
for any real-valued function $b\in{\rm BMO}(\cc)$ and proper functions $f_j$
defined by $Q_j$ with $j\in\nn$; see Lemma \ref{l-cmo-contra} below.
Since a general $A_p(\cc)$ weight is not invariant under translations, besides
Lemma \ref{l-cmo-contra}, we also obtain a  variant of Lemma \ref{l-cmo-contra},
where the geometrical relation of $\{Q_j\}_{j\in\nn}$ are involved;
see Lemma \ref{l-comp contra awy ori} below.
Using Lemmas \ref{l-cmo-contra} and \ref{l-comp contra awy ori} as well as
 an equivalent characterization of CMO$(\cc)$ established by Uchiyama
\cite{Uchiyama78TohokuMathJ}, we give the proof of
Theorem \ref{t-Beurling com compact}(ii) via a contradiction argument.

In Section \ref{s4}, we study the Beltrami equation and  present the proof of
Theorem \ref{t-Bretrami equ} as an application of Theorem \ref{t-Beurling com compact}.
We follow the ideas used in \cite{Iwaniec92} (or \cite{ClopCruz13AASFM}) and apply
some known properties of $\mathcal B$ and the index theory of Fredholm operators.

Finally, we make some conventions on notation. Throughout the article,
we denote by $C$ and $\wz C$ { positive constants} which are independent
of the main parameters, but they may vary from line to line.
Positive constants with subscripts, such as $C_0$ and $\wz C_1$, do not change in
different occurrences. If $f\le Cg$, we then write $f\ls g$ or $g\gs f$
and, if $f \ls g\ls f$, we then write $f\sim g.$

\section{Boundedness characterization of Beurling-Ahlfors commutators}\label{s2.5}

This section is devoted to the proof of Theorem \ref{t-Beurling com bounded}.
Since Theorem \ref{t-Beurling com bounded}(i) is a corollary of
\cite[Theorem 3.4]{KomoriShirai09MathNachr},
it suffices to prove Theorem \ref{t-Beurling com bounded}(ii).
Compared to the method used in \cite{LernerOmbrosiRivera17arxiv,TaoYangYang18arxiv},
our method avoids the use of the so-called local mean oscillation; see also \cite{GuoHeWuYang18arxiv,Janson78ArkMat}.

Here and hereafter, for any $z\in\cc$, square $Q\subset\cc$ and $f\in L^1_{\loc}(\cc)$,
$$Q+z:=\{u+z:\ u\in Q\}$$ and
$$\co (f;Q):=\frac 1{|Q|}\int_{Q}\lf|f(z)-\langle f\rangle_{Q}\r|\,dz
\quad \mathrm{with}\quad \langle f\rangle_{Q}:=\frac{1}{|Q|}\int_Q f(z)\,dz.$$
We first recall the \emph{median value} $\alpha_Q(f)$ in \cite{John65,Stromberg79IUMJ,Journe83,JawerthTorchinsky85JApproxTheory}.
For any real-valued function $f\in L^1_{\rm loc}(\cc)$ and square $Q\subset \cc$,
let $\alpha_Q(f)$ be a real number such that
$$\inf_{c\in\rr}\frac{1}{|Q|}\int_Q|f(z)-c|\,dz$$
is attained. %, where $z:=x+i y\in Q$.
Moreover,  it is known that $\alpha_Q(f)$ satisfies that
\begin{equation}\label{median value-1}
   \lf|\lf\{z\in Q:\ f(z)\geq\alpha_Q(f)\r\}\r|\le \frac{|Q|}2
\end{equation}
and
\begin{equation}\label{median value-2}
   \lf|\lf\{z\in Q:\ f(z)\leq\alpha_Q(f)\r\}\r|\le \frac{|Q|}2;
\end{equation}
see \cite[p.\,30]{Journe83}.

\begin{lem}\label{product set of b}
  Let $b$ be a real-valued measurable function on $\cc$.
  Then, for any square $Q:=Q(z_0,r_0)\subset\cc$ with $z_0\in\cc$ and $r_0\in(0,\infty)$,
  there exist measurable sets $E_1,\,E_2\subset Q$ and
  $F_1,\,F_2\subset\widetilde{Q}:=Q+\widetilde{z_0}$ with
  $\widetilde{z_0}:=4r_0+i4r_0$ such that
\begin{enumerate}
    \item[{\rm(i)}] $Q=E_1\bigcup E_2,\,\widetilde{Q}=F_1\bigcup F_2$
                     and $|F_j|\geq\frac{1}{2}\lf|\widetilde{Q}\r|,\,j\in\{1,2\};$
    \item[{\rm(ii)}] $|b(z)-\alpha_{\widetilde{Q}}(b)|\le|b(z)-b(u)|, \,\,
                     \forall(z,u)\in E_j\times F_j,\,j\in\{1,2\};$
    \item[{\rm(iii)}] for any $(z,u)\in E_j\times F_j$ with $j\in\{1,2\}$,
    both $(x-\zeta)(y-\eta)$ and $b(z)-b(u)$ do not change sign,
    where $z:=x+iy$ and $u:=\zeta+i\eta$ with $x,\,y,\,\zeta,\,\eta\in\rr$.
\end{enumerate}
\end{lem}

\begin{proof}
  For the given squares $Q$ and $\widetilde{Q}$, let
  $$E_1:=\lf\{z\in Q:\ b(z)\geq\alpha_{\widetilde{Q}}(b)\r\}\quad\mathrm{and}
  \quad E_2:=\lf\{z\in Q:\ b(z)\leq\alpha_{\widetilde{Q}}(b)\r\};$$
  $$F_1:=\lf\{u\in\widetilde{Q}:\ b(u)\leq\alpha_{\widetilde{Q}}(b)\r\}\quad\mathrm{and}
  \quad F_2:=\lf\{u\in\widetilde{Q}:\ b(u)\geq\alpha_{\widetilde{Q}}(b)\r\}.$$
  It is easy to see that $\{E_j\times F_j\}_{j=1}^2$ satisfies (iii).
  Then, by \eqref{median value-1} and \eqref{median value-2}, we have
  $|F_j|\geq\frac{1}{2}|\widetilde{Q}|,\,j\in\{1,2\}$, that is, (i) holds true.
  Moreover, for any $(z,u)\in E_j\times F_j,\,j\in\{1,2\}$,
  $$|b(z)-b(u)|=\lf|b(z)-\alpha_{\widetilde{Q}}(b)\r|+\lf|\alpha_{\widetilde{Q}}(b)-b(u)\r|
  \geq\lf|b(z)-\alpha_{\widetilde{Q}}(b)\r|$$
  and hence (ii) holds true.
  This finishes the proof of Lemma \ref{product set of b}.
\end{proof}

Now, we give the proof of Theorem \ref{t-Beurling com bounded}(ii).
\begin{proof}
  To show that $b\in\rm{BMO}(\cc)$, it suffices to show that,
  for any square $Q\subset\cc$, $\co(b;Q)\lesssim 1.$
  Let $Q$ be a square in $\cc$ and $\widetilde{Q}$, $E_j,\,F_j,\,j\in\{1,2\}$ be as in
  Lemma \ref{product set of b}. Since $b$ is real-valued,
  from the H\"{o}lder inequality and the boundedness of $[b,\mathcal{B}]$
  on $L^{p,\,\kappa}_w(\cc)$, we deduce that
  \begin{align*}
    \co(b;Q)&\lesssim\frac{1}{|Q|}\int_Q\lf|b(z)-\alpha_{\widetilde{Q}}(b)\r|\,dz
    \sim\sum_{j=1}^2\frac{1}{|Q|}\int_{E_j}\lf|b(z)-\alpha_{\widetilde{Q}}(b)\r|\,dz\\
    &\lesssim\sum_{j=1}^2\frac{1}{|Q|}\int_{E_j}\int_{F_j}
    \frac{|b(z)-\alpha_{\widetilde{Q}}(b)|}{|Q|}\,du\,dz
    \sim\sum_{j=1}^2\frac{1}{|Q|}\int_{E_j}\int_{F_j}
    \frac{|b(z)-\alpha_{\widetilde{Q}}(b)|}{|z-u|^2}\,du\,dz\\
    &\lesssim\sum_{j=1}^2\frac{1}{|Q|}\int_{E_j}\int_{F_j}
    |b(z)-b(u)|\frac{|(x-\zeta)(y-\eta)|}{|z-u|^4}\,du\,dz\\
    &\sim\sum_{j=1}^2\frac{1}{|Q|}\int_{E_j}\lf|\int_{F_j}
    [b(z)-b(u)]\mathfrak{Im}K_\mathcal{B}(z,u)\,du\r|\,dz\\
    &\lesssim\sum_{j=1}^2\frac{1}{|Q|}\int_{E_j}\lf|\int_{F_j}
    [b(z)-b(u)]K_\mathcal{B}(z,u)\,du\r|\,dz
    \sim\sum_{j=1}^2\frac{1}{|Q|}\int_{E_j}\lf|[b,\mathcal{B}]\chi_{F_j}(z)\r|\,dz\\
    &\lesssim\sum_{j=1}^2\frac{1}{|Q|}\int_{Q}\lf|[b,\mathcal{B}]\chi_{F_j}(z)\r|\,dz
    \lesssim\sum_{j=1}^2\frac{1}{|Q|}\lf\|[b,\mathcal B]\chi_{F_j}\r\|_{L_w^{p,\,\kappa}(\cc)}
    [w(Q)]^{\frac{\kappa-1}{p}}|Q|\\
    &\lesssim\sum_{j=1}^2\lf\|[b,\mathcal B]\r\|_{L_w^{p,\,\kappa}(\cc)\to
    L_w^{p,\,\kappa}(\cc)}\lf\|\chi_{F_j}\r\|_{L_w^{p,\,\kappa}(\cc)}
    [w(Q)]^{\frac{\kappa-1}{p}}\\
    &\lesssim\sum_{j=1}^2\|[b,\mathcal B]\|_{L_w^{p,\,\kappa}(\cc)\to
    L_w^{p,\,\kappa}(\cc)}\lf[w\lf(F_j\r)\r]^{\frac{1-\kappa}{p}} [w(Q)]^{\frac{\kappa-1}{p}}\\
    &\lesssim\|[b,\mathcal B]\|_{L_w^{p,\,\kappa}(\cc)\to
    L_w^{p,\,\kappa}(\cc)}\lf[w\lf(\widetilde{Q}\r)\r]^{\frac{1-\kappa}{p}}
    [w(Q)]^{\frac{\kappa-1}{p}}
    \lesssim\|[b,\mathcal B]\|_{L_w^{p,\,\kappa}(\cc)\to L_w^{p,\,\kappa}(\cc)},
  \end{align*}
  where $\mathfrak{Im}K_\mathcal{B}(z,u)$ denotes the \emph{imaginary part} of
  $K_\mathcal{B}(z,u)$. This finishes the proof of Theorem \ref{t-Beurling com bounded}.
\end{proof}

\section{Compactness characterization of Beurling-Ahlfors commutators}\label{s3}

This section is devoted to the proof of Theorem \ref{t-Beurling com compact}.
We present the proof of Theorem \ref{t-Beurling com compact}(i) in Subsection \ref{s3.1}
and the proof of Theorem \ref{t-Beurling com compact}(ii) in Subsection \ref{s3.2}.

\subsection{Proof of Theorem \ref{t-Beurling com compact}(i)}\label{s3.1}
We first recall  a sufficient condition for  subsets of weighted Morrey spaces
to be relatively compact from \cite{MaoSunWu16ActaMathSin}. Recall that
a subset $\cf$ of $L_w^{p,\,\kappa}(\cc)$ is said to be \emph{totally bounded}
(or \emph{relatively compact}) if the $L_w^{p,\,\kappa}(\cc)$-closure of
$\cf$ is compact.

\begin{lem}\label{l-fre kol}
For any $p\in(1,\infty)$, $\kappa\in(0,1)$ and $w\in A_p(\cc)$, a subset $\cf$ of
$L_w^{p,\,\kappa}(\cc)$ is totally bounded (or relatively compact) if
the set $\cf$ satisfies the following three conditions:
\begin{itemize}
\item[{\rm(i)}] $\cf$ is bounded, namely, $$\sup_{f\in\cf}\|f\|_{L_w^{p,\,\kappa}(\cc)}<\infty;$$
\item[{\rm(ii)}] $\cf$ uniformly vanishes at infinity, namely, for any $\epsilon\in(0,\infty)$,
there exists some positive constant $M$ such that, for any $f\in\cf$,
$$\lf\|f\chi_{\{z\in\cc:\ |z|>M\}}\r\|_{L_w^{p,\,\kappa}(\cc)}<\epsilon;$$

\item[{\rm(iii)}] $\cf$ is uniformly equicontinuous, namely, for any $\epsilon\in(0,\infty)$,
there exists some positive constant $\rho$ such that,
for any $f\in\cf$ and $\xi\in \cc$ with $|\xi|\in[0,\rho)$,
$$\|f(\cdot+\xi)-f(\cdot)\|_{L_w^{p,\,\kappa}(\cc)}<\epsilon.$$
\end{itemize}
\end{lem}

Inspired by \cite{KrantzLi01JMAAb} (see also \cite{ClopCruz13AASFM}),
before we give the proof of Theorem \ref{t-Beurling com compact},
we first establish the boundedness of the maximal operator $\mathcal B_\ast$
of a family of smooth truncated Beurling-Ahlfors transforms
$\{\mathcal B_\eta\}_{\eta\in(0,\infty)}$ as follows.
For $\eta\in(0,\infty)$, let
$$\mathcal B_\eta f(z):=\int_\cc K_{\mathcal B,\,\eta}(z, u)f(u)\,du,$$
where the kernel  
$K_{\mathcal B,\,\eta}(z,u):=K_{\mathcal B}(z,u)\varphi(\frac{|z-u|}\eta)$ with
$K_{\mathcal B}(z,u)=-\frac1\pi\frac1{(z-u)^2}$ and
$\varphi\in C_c^{\infty}(\rr)$ satisfying that 
$\varphi(t)\equiv 0$ for $t\in(-\fz, \frac12)$, $\varphi(t)\in[0,1]$ for $t\in[\frac12, 1]$
and $\varphi(t)\equiv 1$ for $t\in (1, \infty)$.
Let $$[b, \mathcal B_\eta]f(z):=\int_\cc [b(z)-b(u)]K_{\mathcal B,\,\eta}(z, u)f(u)\,du.$$
Then we have the following conclusion. Recall that
the \emph{Hardy-Littlewood maximal operator $\cm$} is defined by setting, for any
$f\in L^1_{\loc}(\cc)$ and $z\in\cc$,
$$\cm f(z):=\sup_{\mathrm{square}\,Q\ni z}\frac{1}{|Q|}\int_Q|f(u)|\,du,$$
where the supremum is taken over all the squares $Q$ of $\cc$ that contain $z$.

\begin{lem}\label{l-approx commuta}
There exists a positive constant $C$ such that,
for any $b\in C^\fz_c(\cc)$, $\eta\in(0,\infty)$, $f\in L^1_{\loc}(\cc)$ and $z\in\cc$,
$$\lf|\lf[b, \mathcal B_\eta\r]f(z)-\lf[b, \mathcal B\r]f(z)\r|
\le C\eta \lf\|\nabla b\r\|_{L^\infty(\cc)} \cm f(z).$$
\end{lem}

\begin{proof}
Let $f\in L^1_{\loc}(\cc)$. For any $z\in\cc$, we have
\begin{align*}
&\lf|\lf[b, \mathcal B_\eta\r]f(z)-\lf[b, \mathcal B\r]f(z)\r|\\
&\quad=\lf|\int_{\eta/2<|z-u|\leq\eta} [b(z)-b(u)]K_{\mathcal B,\,\eta}(z,u)f(u)\,du
    -\int_{|z-u|\leq\eta} [b(z)-b(u)]K_{\mathcal B}(z,u)f(u)\,du\r|\\
&\quad\ls\int_{|z-u|\leq\eta} |b(z)-b(u)|\lf|K_{\mathcal B}(z,u)\r||f(u)|\,du.
\end{align*}
From the smoothness of $b$ and \eqref{cz kernel condition-1}, we deduce that
\begin{align*}
\int_{|z-u|\leq\eta} |b(z)-b(u)|\lf|K_{\mathcal B}(z, u)\r||f(u)|\,du
&\ls \lf\|\nabla b\r\|_{L^\infty(\cc)}\sum_{j=0}^\infty
\int_{\frac\eta{2^{j+1}}<|z-u|\leq\frac\eta{2^j}} \frac{|z-u|}{|z-u|^{2}}|f(u)|\,du\\
&\ls \eta\lf\|\nabla b\r\|_{L^\infty(\cc)}\cm f(z),
\end{align*}
which completes the proof of Lemma \ref{l-approx commuta}.
\end{proof}

In what follows, the\emph{ maximal operator $\mathcal B_{\ast}$} is 
defined by setting, for any suitable function $f$ and $z\in\cc$,
$$\mathcal B_{\ast} f(z):=\sup_{\eta\in(0,\infty)}
\lf|\int_{\cc}K_{\mathcal B,\,\eta}(z, u)f(u)\,du\r|.$$
\begin{lem}\label{l-sub opr bdd}
Let $p\in(1,\infty)$, $\kappa\in(0,1)$ and $w\in A_p(\cc)$.
Then there exists a positive constant $C$ such that,
for any $f\in L_w^{p,\,\kappa}(\cc)$,
$$\lf\|\mathcal B_{\ast} f\r\|_{L_w^{p,\,\kappa}(\cc)}
+\|\cm f\|_{L_w^{p,\,\kappa}(\cc)}\leq C\|f\|_{L_w^{p,\,\kappa}(\cc)}.$$
\end{lem}

\begin{proof}
The boundedness of $\cm$ on $L_w^{p,\,\kappa}(\cc)$ was obtained in
\cite{AraiMizuhara97MathNachr}. We only consider the boundedness of $\mathcal B_\ast$.
The argument is standard and we give the proof briefly.
For any fixed square $Q\subset \cc$ and $f\in L_w^{p,\,\kappa} (\cc)$, we write
$$f:=f_1+f_2:=f\chi_{2Q}+ f\chi_{\cc\setminus 2Q}.$$
Observe that $K_{\mathcal B,\,\eta}$ satisfies \eqref{cz kernel condition-1},
\eqref{cz kernel condition-2} and $f_1\in L_w^{p} (\cc)$.
Then, from the boundedness of $\mathcal B_\ast$ on $L_w^{p} (\cc)$ (see, for example,
\cite[p.\,147, Corollary 7.13]{Duoandikoetxea01}),  the H\"older inequality,
Definition \ref{d-Ap} and $w(2Q)\sim w(Q)$ for any square $Q\subset\cc$, we deduce that
\begin{align*}
&\lf[\int_Q|\mathcal B_\ast f(z)|^pw(z)\,dz\r]^{\frac1p}\\
&\quad\ls \lf[\int_Q|\mathcal B_\ast f_1(z)|^pw(z)\,dz\r]^{\frac1p}
  +\sum_{k=1}^\fz\lf\{\int_Q\lf[\int_{2^{k+1}Q\setminus 2^kQ}
  \frac{|f(u)|}{|z-u|^2}\,du\r]^pw(z)\,dz\r\}^{\frac1p}\\
&\quad\ls \lf[\int_{2Q}|f(z)|^pw(z)\,dz\r]^{\frac1p}+
  \sum_{k=1}^\fz \lf[\frac{w(Q)}{|2^kQ|^p}\lf\{\int_{2^{k+1}Q}|f(u)|[w(u)]^{\frac1p}
  [w(u)]^{-\frac1p}\,du\r\}^p\r]^{\frac1p}\\
&\quad\ls \|f\|_{L_w^{p,\,\kappa} (\cc)}[w(Q)]^{\frac\kappa p}+\sum_{k=1}^\fz \lf\{w(Q)\lf[w\lf(2^kQ\r)\r]^{\kappa-1}\|f\|^p_{L_w^{p,\,\kappa} (\cc)}\r\}^{\frac1p}\\
&\quad\ls \|f\|_{L_w^{p,\,\kappa} (\cc)}w(Q)^{\frac\kappa p}+\sum_{k=1}^\fz \lf\{[w(Q)]^\kappa2^{2k\sigma(\kappa-1)}\|f\|^p_{L_w^{p,\,\kappa} (\cc)}\r\}^{\frac1p}
\ls \|f\|_{L_w^{p,\,\kappa} (\cc)}[w(Q)]^{\frac\kappa p},
\end{align*}
where, in the penultimate inequality, we used the fact that, since $w\in A_p(\cc)$,
it follows that there exist positive constants $C_{(w)}$ and $\sigma\in(0, 1)$ such that,
for any square $Q\subset\cc$ and measurable set $E\subset Q$,
\begin{equation}\label{e-sigma defn}
\frac{w(E)}{w(Q)}\leq C_{(w)} \lf(\frac{|E|}{|Q|}\r)^\sigma.
\end{equation}
This finishes the proof of Lemma \ref{l-sub opr bdd}.
\end{proof}

\begin{proof}[Proof of Theorem \ref{t-Beurling com compact}(i)]
When $b\in \cmoc$, for any $\varepsilon\in(0,\infty)$,
there exists $b^{(\varepsilon)}\in C^\fz_c(\cc)$ such that
$\|b-b^{(\varepsilon)}\|_{\bmoc}<\varepsilon.$
Then, by the boundedness of $[b, \mathcal B]$ on $L_w^{p,\,\kappa} (\cc)$
(see \cite[Theorem 3.4]{KomoriShirai09MathNachr}), we obtain
\begin{align*}
\lf\|\lf[b,\mathcal B\r]f-\lf[b^{(\varepsilon)},\mathcal B\r]f\r\|_{L_w^{p,\,\kappa} (\cc)}
=\lf\|\lf[b-b^{(\varepsilon)},\mathcal B\r]f\r\|_{L_w^{p,\,\kappa} (\cc)}
\lesssim\lf\|b-b^{(\varepsilon)}\r\|_{\bmoc}\|f\|_{L_w^{p,\,\kappa} (\cc)}
\le\varepsilon\|f\|_{L_w^{p,\,\kappa} (\cc)}.
\end{align*}
Moreover, from Lemmas \ref{l-approx commuta} and \ref{l-sub opr bdd},
we deduce that
$$\lim_{\eta\to0}\lf\|\lf[b, \mathcal B_\eta\r]-
  \lf[b, \mathcal B\r]\r\|_{L_w^{p,\,\kappa} (\cc)\to L_w^{p,\,\kappa} (\cc)}=0.$$
Thus, it suffices to show that, for any $b\in C^\fz_c(\cc)$ and $\eta\in(0,\infty)$
small enough, $[b,\,\mathcal B_\eta]$ is a compact operator on $L_w^{p,\,\kappa} (\cc)$.
From the definition of compact operators, to show $[b,\,\mathcal B_\eta]$ is compact
on $L_w^{p,\,\kappa} (\cc)$, it suffices to show that, for any bounded subset
$\cf\subset L_w^{p,\,\kappa} (\cc)$, $[b,\,\mathcal B_\eta]\cf$ is relatively compact.
It follows from Lemma \ref{l-fre kol} that we only need to show that
$[b,\,\mathcal B_\eta]\cf$ satisfies the conditions (i) through (iii) of Lemma \ref{l-fre kol}.

We first point out that, by \cite[Theorem 3.4]{KomoriShirai09MathNachr}
and the fact that $b\in\bmoc$, we know that
$[b,\,\mathcal B_\eta]$ is bounded on $L_w^{p,\,\kappa} (\cc)$ for the given $p\in(1,\infty)$,
$\kappa\in(0,1)$ and $w\in A_p(\cc)$, which implies that
$[b,\,\mathcal B_\eta]\cf$ satisfies condition (i) of Lemma \ref{l-fre kol}.

Next, since $b\in C^\fz_c(\cc)$, we may further assume
$\|b\|_{L^{\infty}(\cc)}+\|\nabla b\|_{L^{\infty}(\cc)}=1$.
Observe that there exists a positive constant $R_0$ such that $\supp (b)\subset Q(0,R_0)$.
Let $M\in(10R_0,\infty)$. Thus, for any $u\in Q(0,R_0)$ and $z\in\cc$ with $|z|\in(M,\infty)$,
we have $|z-u|\sim |z|$.
Then, by \eqref{cz kernel condition-1} and the H\"{o}lder inequality,
we conclude that
\begin{align*}
\lf|\lf[b,\,\mathcal B_\eta\r]f(z)\r|&\le\int_\cc|b(z)-b(u)|\lf|K_{\cb,\,\eta}(z, u)\r|
  |f(u)|\,du
\ls\|b\|_{L^{\infty}(\cc)}\int_{Q(0,\,R_0)}\frac{|f(u)|}{|z-u|^2}\,du\\
&\ls\frac{1}{|z|^2}\|b\|_{L^{\infty}(\cc)}\lf[\int_{Q(0,\,R_0)}|f(u)|^pw(u)\,du\r]^
 {\frac1p}\lf\{\int_{Q(0,\,R_0)}[w(u)]^{-\frac{p'}p}\,du\r\}^{\frac1{p'}}\\
&\ls\frac{1}{|z|^2}\|f\|_{L_w^{p,\,\kappa} (\cc)}
 \lf[w(Q(0,\,R_0))\r]^{\frac{\kappa-1}p}|Q(0,\,R_0)|.
\end{align*}
Therefore, for any fixed square $U:=Q(\wz z,\,\wz r)\subset \cc$, we have
\begin{align*}
&\frac{1}{[w(U)]^\kappa}\int_{U\cap\{z\in\cc:\ |z|>M\}}
    \lf|\lf[b,\,\mathcal B_\eta\r]f(z)\r|^p w(z)\,dz\\
&\quad\ls\frac{\|f\|^p_{L_w^{p,\,\kappa} (\cc)}[w(Q(0,R_0))]^{\kappa-1}
    |Q(0,R_0)|^p}{[w(U)]^\kappa}\sum_{j=0}^\infty
    \frac{w(U\cap\{z\in\cc:\ 2^jM<|z|\leq 2^{j+1}M\})}{|2^jM|^{2p}}\\
&\quad\ls \|f\|^p_{L_w^{p,\,\kappa} (\cc)}[w(Q(0,R_0))]^{\kappa-1}|Q(0,R_0)|^p
   \sum_{j=0}^\infty\frac{[w(Q(0, 2^jM))]^{1-\kappa}}{|2^jM|^{2p}}\\
&\quad\ls\|f\|^p_{L_w^{p,\,\kappa} (\cc)}[w(Q(0,R_0))]^{\kappa-1}|Q(0,R_0)|^p
   \frac{[w(Q(0, M))]^{1-\kappa}}{M^{2p}}\sum_{j=0}^\infty\frac{2^{2jp(1-\kappa)}}{2^{2jp}}
   \ls\lf(\frac{R_0}{M}\r)^{2p}\|f\|^p_{L_w^{p,\,\kappa} (\cc)},
\end{align*}
where, in the penultimate inequality, we used the fact that, if $w\in A_p(\cc)$
for some $p\in(1,\fz)$, then, for any square $Q\subset \cc$ and $t\in(1,\infty)$,
\begin{align}\label{Ap double}
w(tQ)\ls t^{2p}w(Q).
\end{align}
Thus, we conclude that
$$\lf\|\lf([b,\mathcal B_\eta]f\r)\chi_{\{z\in\cc:\ |z|>M\}}\r\|_{L_w^{p,\,\kappa}(\cc)}
  \ls \lf(\frac{R_0}{M}\r)^{2}\|f\|_{L_w^{p,\,\kappa} (\cc)}.$$
Therefore, condition (ii) of Lemma \ref{l-fre kol} holds true for
$[b,\mathcal B_\eta]\mathcal{F}$ with $M$ large enough.

It remains to prove that $[b,T_\eta]\mathcal{F}$ also satisfies condition (iii) of
Lemma \ref{l-fre kol}.
Let $\eta$ be a fixed positive constant small enough and $\xi\in\cc$ with $|\xi|\in(0,\eta/8)$.
Then, for any $z\in\cc$, we have
\begin{align*}
&\lf[b, \mathcal B_\eta\r]f(z)-\lf[b, \mathcal B_\eta\r]f(z+\xi)\\
&\quad=[b(z)-b(z+\xi)]\int_{\cc}K_{\mathcal B,\,\eta}(z, u)f(u)\,du\\
&\quad\quad +\int_{\cc}\lf[K_{\mathcal B,\,\eta}(z, u)-
  K_{\mathcal B,\,\eta}(z+\xi, u)\r][b(z+\xi)-b(u)]f(u)\,du\\
&\quad=:\sum_{i=1}^2{\rm L}_i(z).
\end{align*}

Since $b\in C^\fz_c(\cc)$, it follows that, for any $z\in\cc$,
\begin{align*}
|{\rm L}_1(z)|=|b(z)-b(z+\xi)|\lf|\int_{\cc}K_{\mathcal B,\,\eta}(z, u)f(u)\,du\r|
 \ls |\xi| \lf\|\nabla b\r\|_{L^{\infty}(\cc)}\mathcal B_\ast(f)(z).
\end{align*}
Then Lemma \ref{l-sub opr bdd} implies
$\|{\rm L}_1\|_{L_w^{p,\,\kappa} (\cc)}\ls \|f\|_{L_w^{p,\,\kappa} (\cc)}$.

To estimate ${\rm L_2}(z)$, we first observe that
$K_{\mathcal B,\,\eta}(z,u)=0$, $K_{\mathcal B,\,\eta}(z+\xi, u)=0$
for any $z$, $u$, $\xi\in\cc$ with $|z-u|\in(0,\eta/ 4)$ and $|\xi|\in(0, \eta/8)$.
Moreover, by the definition of $K_{\mathcal B,\,\eta}(z,u)$ and \eqref{cz kernel condition-2},
we know that, for any $z$, $u$, $\xi\in\cc$ with $|z-u|\in[\eta/4,\infty)$,
$$\lf|K_{\mathcal B,\,\eta}(z,u)-K_{\mathcal B,\,\eta}(z+\xi,u)\r|\ls \frac{|\xi|}{|z-u|^{3}}.$$
This in turn implies that, for any $z\in\cc$,
\begin{align*}
|{\rm L}_2(z)|&\ls |\xi|\int_{|z-u|>\eta/4}\frac{|f(u)|}{|z-u|^{3}}\,\,du
\ls\sum_{k=0}^\fz\frac{|\xi|}{(2^k\eta)^{3}}\int_{2^k\eta/4<|z-u|\le2^{k+1}\eta/4}|f(u)|\,du\\
&\ls\sum_{k=0}^\fz\frac{|\xi|}{2^k\eta}\frac{1}{(2^k\eta)^2}\int_{Q(z,\,2^{k+1}\eta/4)}|f(u)|\,du
\ls\frac{|\xi|}\eta \cm f(z).
\end{align*}
Then, by the boundedness of $\cm$ on $L_w^{p,\,\kappa} (\cc)$, we obtain
$$\|{\rm L}_2\|_{L_w^{p,\,\kappa} (\cc)}\ls\frac{|\xi|}\eta\|f\|_{L_w^{p,\,\kappa} (\cc)}.$$
Combining the estimates of ${\rm L_i}(z)$, $i\in\{1, 2\}$, we conclude that
$[b,\,\mathcal B_\eta]\mathcal{F}$ satisfies condition (iii) of Lemma \ref{l-fre kol}.
Thus, $[b,\,\mathcal B_\eta]$ is a compact operator for any $b\in C^\fz_c(\cc)$.
This finishes the  proof of Theorem \ref{t-Beurling com compact}(i).
\end{proof}

\subsection{Proof of Theorem \ref{t-Beurling com compact}(ii)}\label{s3.2}

We begin with recalling an equivalent characterization of $\cmoc$ from
\cite[p.\,166, Lemma]{Uchiyama78TohokuMathJ}.
In what follows, the \emph{symbol} $a\to0^+$ means that $a\in(0,\fz)$ and $a\to0$.
\begin{lem}\label{l-cmo char}
Let $f\in\bmoc$.
Then $f\in\cmoc$ if and only if $f$ satisfies the following three conditions:
\begin{itemize}
  \item [{\rm(i)}] $$\lim_{a\to0^+}\sup_{|Q|=a}\co(f; Q)=0;$$
  \item [{\rm(ii)}]$$\lim_{a\to\fz}\sup_{|Q|=a}\co(f; Q)=0;$$
  \item [{\rm (iii)}]for any square $Q\subset\cc$, $$\lim_{z\to\fz}\co(f; Q+z)=0.$$
\end{itemize}
\end{lem}

Next, we establish a lemma for the upper and the lower bounds of integrals of
$[b,\,\mathcal B]f_j$ on certain squares $Q_j$ in $\cc$ for any $j\in\nn$.
By the choice of $\alpha_Q(f)$ as in Lemma \ref{product set of b},
it is easy to show that, for any $f\in L^1_{\rm loc}(\cc)$ and square $Q\subset \cc$,
\begin{equation}\label{equi osci}
\co(f;Q)\sim\frac1{|Q|}\int_Q\lf|f(z)-\az_Q(f)\r|\,dz
\end{equation}
with the equivalent positive constants independent of $f$ and $Q$.

\begin{lem}\label{l-cmo-contra}
Let $p\in(1,\infty)$, $\kappa\in(0,1)$ and $w\in A_p(\cc)$.
Suppose that $b\in{\rm BMO}(\cc)$ is a real-valued function with
$\|b\|_{{\rm BMO}(\cc)}=1$ and there exist $\dz\in(0, \fz)$ and a sequence
$\{Q_j\}_{j\in\nn}:=\{Q(z_j, r_j)\}_{j\in\nn}$ of squares in $\cc$,
with $\{z_j\}_{j\in\nn}\subset\cc$ and $\{r_j\}_{j\in\nn}\subset(0,\infty)$,
such that, for any $j\in\nn$,
\begin{equation}\label{lower bdd osci}
\co(b; Q_j)>\dz.
\end{equation}
Then there exist real-valued functions $\{f_j\}_{j\in\nn}\subset L_w^{p,\,\kappa}(\cc)$,
positive constants $K_0$ large enough, $\wz C_0$, $\wz C_1$ and $\wz C_2$
such that, for any $j\in\nn$ and integer $k\ge K_0$,
$\|f_j\|_{L_w^{p,\,\kappa}(\cc)}\le \wz C_0$,
\begin{equation}\label{lower upper lpbdd riesz comm}
\int_{Q_j^k}\lf|\lf[b, \mathcal B\r]f_j(z)\r|^pw(z)\,dz\geq\wz C_1
\frac{\dz^p}{3^{2kp}}\lf[w\lf(Q_j\r)\r]^{\kappa-1} w\lf(3^kQ_j\r),
\end{equation}
where $Q_j^k:=3^{k-1}Q_j+3^{k}r_j\vec{e}$ and $\vec e=(1, 0)$ is the unit vector of $x$-axis,
and
\begin{equation}\label{lower upper lpbdd riesz comm2}
\int_{3^{k+1}Q_j\setminus 3^k Q_j}\lf|\lf[b, \mathcal B\r]f_j(z)\r|^pw(z)\,dz\le \wz C_2\frac1{3^{2kp}}\lf[w\lf(Q_j\r)\r]^{\kappa-1}w\lf(3^{k}Q_j\r).
\end{equation}
\end{lem}

\begin{proof}
For each $j\in\nn$, define the function $f_j$ as follows:
\begin{equation*}
f^{(1)}_j:=\chi_{Q_{j,\,1}}-\chi_{Q_{j,\,2}}:=\chi_{\{z\in Q_j:\  b(z)>\az_{Q_j}(b)\}}-
\chi_{\{z\in Q_j:\  b(z)<\az_{Q_j}(b)\}},\quad f^{(2)}_j:=a_j\chi_{Q_j}
\end{equation*}
and
$$f_j:=\lf[w\lf(Q_j\r)\r]^{\frac{\kappa-1}p}\lf[f^{(1)}_j-f^{(2)}_j\r],$$
 where $Q_j$ is as in the assumption of Lemma \ref{l-cmo-contra} and $a_j\in\rr$ is a
 constant such that
 \begin{equation}\label{fj proper-2}
 \int_\cc f_j(z)\,dz=0.
  \end{equation}
Then, by the definition of $a_j$, \eqref{median value-1} and \eqref{median value-2},
we conclude that $|a_j|\le 1/2$.
Moreover, we also have $\supp(f_j)\subset Q_j$ and, for any $z\in Q_j$,
\begin{equation}\label{fj proper-1}
 f_j(z)\lf[b(z)-\az_{Q_j}(b)\r]\ge0.
 \end{equation}
Moreover, since $|a_j|\le1/2$, we deduce that, for any $z\in (Q_{j,\,1}\cup Q_{j,\,2})$,
  \begin{equation}\label{fj proper-3}
   \lf|f_j(z)\r|\sim \lf[w\lf(Q_j\r)\r]^{\frac{\kappa-1}p}
  \end{equation}
and hence
$$\lf\|f_j\r\|_{L_w^{p,\,\kappa}(\cc)}\ls\sup_{P\subset\cc}
\lf\{\frac{w(P\cap Q_j)}{[w(P)]^\kappa}\r\}^{\frac1p}\lf[w\lf(Q_j\r)\r]^{\frac{\kappa-1}p}
\ls \sup_{P\subset\cc}\lf[w(P\cap Q_j)\r]^{\frac{1-\kappa}p}
\lf[w\lf(Q_j\r)\r]^{\frac{\kappa-1}p}\ls 1.$$

Observe that
\begin{equation}\label{com equiv}
\lf[b, \mathcal B\r]f=\lf[b-\az_{Q_j}(b)\r]\mathcal B(f)-
  \mathcal B\lf(\lf[b-\az_{Q_j}(b)\r]f\r).
\end{equation}
Moreover, for any $k\in\nn$, we have
\begin{equation}\label{ijk set inclu}
3^{k-1}Q_j\subset4Q_j^k\subset3^{k+1}Q_j
\end{equation}
and  hence
\begin{equation}\label{ijk meas}
w\lf(Q_j^k\r)\sim w\lf(3^{k}Q_j\r).
\end{equation}

We now prove inequality \eqref{lower upper lpbdd riesz comm}. By
\eqref{cz kernel condition-2}, \eqref{fj proper-2}, \eqref{fj proper-3}
and the fact that $|z-z_j|\sim|z-\xi|$ for any $z\in Q_j^k$ with integer $k\geq2$
and $\xi\in Q_j$, we conclude that, for any $z\in Q_j^k$,
\begin{align}\label{upper bdd riesz ope}
\lf|\lf[b(z)-\az_{Q_j}(b)\r]\mathcal B(f_j)(z)\r|
&=\lf|b(z)-\az_{Q_j}(b)\r|\lf|\int_{Q_j}\lf[K_{\mathcal B}(z- \xi)-
   K_{\mathcal B}(z- z_j)\r]f_j(\xi)\,d\xi\r|\\
&\le\lf|b(z)-\az_{Q_j}(b)\r|\int_{Q_j}\lf|K_{\mathcal B}(z- \xi)-
   K_{\mathcal B}(z-z_j)\r|\lf|f_j(\xi)\r|\,d\xi\noz\\
&\ls \lf[w\lf(Q_j\r)\r]^{\frac{\kappa-1}p}\lf|b(z)-\az_{Q_j}(b)\r|\int_{Q_j}
\frac{|\xi-z_j|}{|z-z_j|^3}\,d\xi\noz\\
&\ls \lf[w\lf(Q_j\r)\r]^{\frac{\kappa-1}p}r^3_j\frac{|b(z)-\az_{Q_j}(b)|}{|z-z_j|^3}\noz
\ls \frac{[w(Q_j)]^{\frac{\kappa-1}p}}{3^{3k}}\lf|b(z)-\az_{Q_j}(b)\r|.
\end{align}
Moreover, by the well-known John-Nirenberg inequality and $\|b\|_{{\rm BMO}(\cc)}=1$,
we know that,  for each $k\in \nn$ and square $Q\subset\cc$,
\begin{align}\label{b-bmo-bdd}
\int_{3^{k+1}Q}\lf|b(z)-\az_{Q}(b)\r|^p\,dz
&\ls \int_{3^{k+1}Q}\lf|b(z)-\az_{3^{k+1}Q}(b)\r|^p\,dz
+\lf|3^{k+1}Q\r|\lf|\az_{3^{k+1}Q}(b)-\az_{Q}(b)\r|^p\\
&\ls k^p\lf|3^kQ\r|,\noz
\end{align}
where the last inequality is deduced from the fact that
\begin{align*}
  \lf|\az_{3^{k+1}Q}(b)-\az_{Q}(b)\r|\ls
  \lf|\alpha_{3^{k+1}Q}(b)-\langle b\rangle_{3^{k+1}Q}\r|
  +\lf|\langle b\rangle_{3^{k+1}Q}-\langle b\rangle _{Q}\r|
  +\lf|\langle b\rangle_{Q}-\alpha_{Q}(b)\r|
  \ls k.
\end{align*}
Since $w\in A_p(\cc)$, it follows that there exists $\ez\in(0, \fz)$
such that the reverse H\"older inequality
$$\left[\frac1{|Q|}\int_Q w(z)^{1+\ez}\,dz\r]^{\frac1{1+\ez}}\ls \frac1{|Q|}\int_Q w(z)\,dz$$
holds true for any square $Q\subset \cc$. By this fact, the H\"older inequality,
 \eqref{b-bmo-bdd}, \eqref{ijk set inclu} and \eqref{upper bdd riesz ope}, we conclude that
there exists a positive constant $\wz C_3$ such that, for any $k\in \nn$,
\begin{align}\label{upper bdd com}
&\int_{Q_j^k}\lf|\lf[b(z)-\az_{Q_j}(b)\r]\mathcal B(f_j)(z)\r|^pw(z)\,dz\\
&\quad\ls \frac{[w(Q_j)]^{\kappa-1}}{3^{3kp}}\int_{3^{k+1}Q_j}\lf|b(z)-\az_{Q_j}(b)\r|^pw(z)\,dz\noz\\
&\quad\ls \frac{[w(Q_j)]^{\kappa-1}}{3^{3kp}}\lf|3^{k}Q_j\r|\lf[\frac1{|3^{k+1}Q_j|}
\int_{3^{k+1}Q_j}\lf|b(z)-\az_{Q_j}(b)\r|^{p(1+\ez)'}\,dz\r]^{\frac1{(1+\ez)'}}\noz\\
&\quad\quad\times\lf[\frac1{|3^{k+1}Q_j|}\int_{3^{k+1}Q_j}w(z)^{1+\ez}\,dz\r]
  ^{\frac1{1+\ez}}\noz\\
&\quad\ls \frac{k^p}{3^{3kp}}\lf[w\lf(Q_j\r)\r]^{\kappa-1}w\lf(3^{k+1}Q_j\r)
\leq \wz C_3\frac{k^p}{3^{3kp}}\lf[w\lf(Q_j\r)\r]^{\kappa-1}w\lf(3^{k}Q_j\r)\noz.
\end{align}

Next, observing that, for any $z:=x+iy\in Q_j^k$ with integer $k\ge 2$ and
$\xi:=\zeta+i \eta\in Q_j$, we have
$$x-\zeta\ge \lf(2\cdot3^{k-1}-1\r)r_j,\quad|y-\eta|\le 3^{k-1}r_j\quad \mathrm{and} \quad
x-\zeta\sim |z-\xi|\sim|z-z_j|.$$
By this, together with  \eqref{fj proper-1}, \eqref{fj proper-3},
\eqref{equi osci} and \eqref{lower bdd osci},
we conclude that, for any $z\in Q_j^k$,
\begin{align*}
\lf|\mathcal B\lf(\lf[b-\az_{Q_j}(b)\r]f_j\r)(z)\r|
&=\frac1\pi\lf|\int_{Q_{j,\,1}\cup Q_{j,\,2}}
\frac{[b(\xi)-\az_{Q_j}(b)]f_j(\xi)}{(z-\xi)^2}\,d\xi\r|\\
&=\frac1\pi\lf|\int_{Q_{j,\,1}\cup Q_{j,\,2}}
\frac{[(x-\zeta)^2-(y-\eta)^2][b(\xi)-\az_{Q_j}(b)]f_j(\xi)}{|z-\xi|^4}\,d\xi\r.\\
&\quad\quad-2i\lf.\int_{Q_{j,\,1}\cup Q_{j,\,2}}
\frac{(x-\zeta)(y-\eta)[b(\xi)-\az_{Q_j}(b)]f_j(\xi)}{|z-\xi|^4}\,d\xi\r|\\
&\gs\int_{Q_{j,\,1}\cup Q_{j,\,2}}
\frac{(x-\zeta)^2|[b(\xi)-\az_{Q_j}(b)]f_j(\xi)|}{|z-\xi|^4}\,d\xi\\
&\gs\frac{1}{|z-z_j|^2}\lf[w\lf(Q_j\r)\r]^{\frac{\kappa-1}p}
\int_{Q_j}\lf|b(\xi)-\az_{Q_j}(b)\r|\,d\xi
\gs \frac{\dz}{3^{2k}} \lf[w\lf(Q_j\r)\r]^{\frac{\kappa-1}p},
\end{align*}
where $z:=x+iy$ and $\xi:=\zeta+i\eta$ with $x,\,y,\,\zeta,\,\eta\in\rr$.
From this and \eqref{ijk meas}, we deduce that there exists a positive constant $\wz C_4$ such that
\begin{align}\label{low bdd com}
\quad\quad\int_{Q_j^k}\lf|\mathcal B\lf(\lf[b-\az_{Q_j}(b)\r]f_j\r)(z)\r|^pw(z)\,dz
&\gs \frac{\dz^p}{3^{2kp}}\lf[w\lf(Q_j\r)\r]^{\kappa-1} w\lf(Q_j^k\r)\\
&\ge \wz C_4\frac{\dz^p}{3^{2kp}}\lf[w\lf(Q_j\r)\r]^{\kappa-1} w\lf(3^kQ_j\r)\noz.
\end{align}
Take $K_0\in(0,\infty)$ large enough such that, for any integer $k\ge K_0$,
$$\wz C_4\frac{\dz^p}{2^{p-1}}-\wz C_3\frac{k^p}{3^{kp}}\ge \wz C_4\frac{\dz^p}{2^p}.$$
From this, \eqref{com equiv},  \eqref{upper bdd com} and \eqref{low bdd com},
we further deduce that
\begin{align*}
&\int_{Q_j^k}\lf|[b,\,\mathcal B]f_j(z)\r|^pw(z)\,dz\noz\\
&\quad\ge\frac1{2^{p-1}}\int_{Q_j^k}\lf|\mathcal B\lf(\lf[b-\az_{Q_j}(b)\r]f_j\r)(z)\r|^pw(z)\,dz
-\int_{Q_j^k}\lf|\lf[b(z)-\az_{Q_j}(b)\r]\mathcal B(f_j)(z)\r|^pw(z)\,dz\noz\\
&\quad\ge\lf(\wz C_4\frac{\dz^p}{2^{p-1}}-\wz C_3\frac{k^p}{3^{kp}}\r)
\frac1{3^{2kp}}\lf[w\lf(Q_j\r)\r]^{\kappa-1} w\lf(3^kQ_j\r)
\ge \frac{\wz C_4}{2^p}
\frac{\dz^p}{3^{2kp}}\lf[w\lf(Q_j\r)\r]^{\kappa-1} w\lf(3^kQ_j\r).
\end{align*}
This shows inequality \eqref{lower upper lpbdd riesz comm}.

Now, we show the inequality \eqref{lower upper lpbdd riesz comm2}.
From $\supp(f_j)\subset Q_j$, \eqref{cz kernel condition-1}, \eqref{fj proper-3},
\eqref{equi osci} and $\|b\|_{{\rm BMO}(\cc)}=1$,
we deduce that, for any $z\in 3^{k+1}Q_j\setminus 3^{k}Q_j$,
\begin{align*}
\lf|\mathcal B\lf(\lf[b-\az_{Q_j}(b)\r]f_j\r)(z)\r|\ls \lf[w\lf(Q_j\r)\r]^{\frac{\kappa-1}p}
\int_{Q_j}\frac{|b(\xi)-\az_{Q_j}(b)|}{|z-\xi|^2}\,d\xi
\ls \lf[w\lf(Q_j\r)\r]^{\frac{\kappa-1}p}\frac1{3^{2k}}.
\end{align*}
Therefore, by \eqref{upper bdd com} (which holds true with $Q^k_j$ replaced by
$3^{k+1}Q_j\setminus 3^kQ_j$), we know that, for any integer $k\ge K_0$,
\begin{align*}
&\int_{3^{k+1}Q_j\setminus 3^kQ_j}\lf|[b,\,\mathcal B]f_j(z)\r|^pw(z)\,dz\\
&\quad\ls\int_{3^{k+1}Q_j\setminus 3^kQ_j}\lf|\mathcal B
   \lf(\lf[b-\az_{Q_j}(b)\r]f_j\r)(z)\r|^pw(z)\,dz
   +\int_{3^{k+1}Q_j\setminus 3^kQ_j}\lf|\lf[b(z)-\az_{Q_j}(b)\r]
   \mathcal B(f_j)(z)\r|^pw(z)\,dz\\
&\quad\ls \lf[w\lf(Q_j\r)\r]^{\kappa-1}\frac1{3^{2kp}}w\lf(3^{k+1}Q_j\r)+\frac{k^p}{3^{3kp}}
  \lf[w\lf(Q_j\r)\r]^{\kappa-1}w\lf(3^{k}Q_j\r)
\ls \lf[w\lf(Q_j\r)\r]^{\kappa-1}\frac1{3^{2kp}}w\lf(3^{k}Q_j\r).
\end{align*}
This finishes the proof of Lemma \ref{l-cmo-contra}.
\end{proof}

Lemma \ref{l-cmo-contra} is sufficient to derive the necessity of the compactness
of Calder\'on-Zygmund commutators in unweighted case;
see, for example, \cite{TaoYangYang18arxiv}.
For weighted case, since a general weight $w\in A_p(\cc)$ is not invariant under
translations, we also need the following Lemma \ref{l-comp contra awy ori}
to deal with some tricky situations.

\begin{lem}\label{l-comp contra awy ori}
Let $p\in(1,\infty)$, $\kappa\in(0,1)$, $w\in A_p(\cc)$,
$b\in {\rm BMO}(\cc),\,\delta,\,K_0\in(0,\infty)$,
$\{f_j\}_{j\in\nn}$ and $\{Q_j\}_{j\in\nn}$ be as in Lemma \ref{l-cmo-contra}.
Assume that $\{Q_j\}_{j\in\nn}:=\{Q(z_j, r_j)\}_{j\in\nn}$ also satisfies
the following two conditions:
\begin{itemize}
  \item [{\rm(i)}] $\forall \ell,\,m\in\nn$ and $\ell\neq m$,
    \begin{equation}\label{I-j-pro}
    3C_1Q_{\ell}\bigcap 3C_1Q_m=\emptyset.
    \end{equation}
    where $C_1:=3^{K_1}> C_2:=3^{K_0}$ for some $K_1\in\nn$ large enough.
  \item [{\rm(ii)}] $\{r_j\}_{j\in\nn}$ is either non-increasing or non-decreasing
    in $j$, or there exist positive constants $C_{\mathrm{min}}$
    and $C_{\mathrm{max}}$ such that, for any $j\in\nn$,
    $$C_{\mathrm{min}}\leq r_j\leq C_{\mathrm{max}}.$$
\end{itemize}
Then there exists a positive constant $C$ such that, for any $j,\,m\in\nn$,
$$\lf\|[b,\mathcal B]f_j- [b,\mathcal B]f_{j+m}\r\|_{L_w^{p,\,\kappa}(\cc)}\geq C.$$
\end{lem}

\begin{proof}
Without loss of generality, we may assume that $\|b\|_{{\rm BMO}(\cc)}\!=1$
and $\{r_j\}_{j\in\nn}$ is non-increasing.
Let $\{f_j\}_{j\in\nn}$, $\wz C_1$, $\wz C_2$ be as in Lemma \ref{l-cmo-contra}
associated with $\{Q_j\}_{j\in\nn}$.
Recall that, for any $w\in A_p(\cc)$ with $p\in(1,\fz)$, there exists $p_0\in (1, p)$
such that $w\in A_{p_0}(\cc)$.
By \eqref{lower upper lpbdd riesz comm}, \eqref{ijk meas}, \eqref{e-sigma defn}
and \eqref{Ap double} with $w\in A_{p_0}(\cc)$, we find that, for any $j\in\nn$,
\begin{align}\label{riesz-comm-lower-bdd-on-A2Ij}
&\lf[\int_{C_1 Q_j}\lf|[b,\mathcal B]f_j(z)\r|^pw(z)\,dz\r]
  ^{1/p}\lf[w\lf(C_1 Q_j\r)\r]^{-\kappa/p}\\
&\quad\geq \lf[w\lf(C_1 Q_j\r)\r]^{-\kappa/p}\lf\{\sum_{k=K_0}^{K_1-2}\int_{Q_{j}^k}
\lf|[b,\mathcal B]f_j(z)\r|^pw(z)\,dz\r\}^{1/p}\noz\\
&\quad\ge \lf[w\lf(C_1 Q_j\r)\r]^{-\kappa/p}\lf\{\sum_{k=K_0}^{K_1-2}\wz C_1 \delta^p
\frac{[w(Q_{j})]^{\kappa -1}w(3^{k}Q_{j})}{3^{2pk}}\r\}^{1/p}\noz\\
&\quad\gtrsim \lf[w\lf(C_1 Q_j\r)\r]^{-\kappa/p}\lf\{\sum_{k=K_0}^{K_1-2} \delta^p
\frac{[w(Q_j)]^\kappa}{3^{2(p-\sigma)k}}\r\}^{1/p}\noz\\
&\quad\ge C_3 C_1^{-\frac{2\kappa}pp_0}\lf[w\lf(Q_j\r)\r]^{-\kappa/p}\dz
 \lf[w\lf(Q_{j}\r)\r]^{\kappa/p}= C_3\dz C_1^{-\frac{2\kappa}pp_0}\noz
\end{align}
for some positive constant $C_3$ independent of $\delta$ and $C_1$.
We next prove that, for any $j,\,m\in\nn$,
\begin{equation}\label{riesz-comm-upper-bdd-on-A2Ij}
\lf[\int_{C_1 Q_j}\lf|[b,\mathcal B]f_{j+m}(z)\r|^pw(z)\,dz\r]^{1/p}
\lf[w\lf(C_1 Q_j\r)\r]^{-\kappa/p}\le \frac 12C_3\dz C_1^{-\frac{2\kappa}pp_0}.
\end{equation}
Indeed, since $\supp (f_{j+m})\subset Q_{j+m}$, from \eqref{fj proper-3}, \eqref{equi osci},
\eqref{I-j-pro} and $\|b\|_{{\rm BMO}(\cc)}=1$, it follows that, for any $z\in C_1Q_j$,
\begin{align*}
\lf|\mathcal B\lf(\lf[b-\alpha_{Q_{j+m}}(b)\r]f_{j+m}\r)(z)\r|
&\ls \lf[w\lf(Q_{j+m}\r)\r]^{\frac{\kappa-1}p}\int_{Q_{j+m}}|K_{\mathcal B}(z-\xi)|
  \lf|b(z)-\alpha_{Q_{j+m}}(b)\r|\,d\xi\\
&\ls \lf[w\lf(Q_{j+m}\r)\r]^{\frac{\kappa-1}p}\frac{r_{j+m}^2}{|z_j-z_{j+m}|^2}
\end{align*}
and hence
\begin{align}\label{riesz-comm-upper-bdd-on-A2Ij-i}
&\lf\{\int_{C_1 Q_j}\lf|\mathcal B\lf(\lf[b-\alpha_{Q_{j+m}}(b)\r]f_{j+m}\r)(z)\r|^p
w(z)\,dz\r\}^{1/p}\lf[w\lf(C_1 Q_j\r)\r]^{-\kappa/p}\\
&\quad\ls \lf[w\lf(Q_{j+m}\r)\r]^{\frac{\kappa-1}p}\frac{r_{j+m}^2}{|z_j-z_{j+m}|^2}
\lf[w\lf(C_1 Q_j\r)\r]^{\frac{1-\kappa}p}\noz\\
&\quad\ls \lf[w\lf(Q_{j+m}\r)\r]^{\frac{\kappa-1}p}\frac{r_{j+m}^2}{|z_j-z_{j+m}|^2}
\lf[w\lf(\frac{|z_j-z_{j+m}|}{r_{j+m}} Q_{j+m}\r)\r]^{\frac{1-\kappa}p}\noz\\
&\quad\ls \frac{r_{j+m}^2}{|z_j-z_{j+m}|^2}\lf(\frac{|z_j-z_{j+m}|}{r_{j+m}}\r)
  ^{2\frac{1-\kappa}p p_0}
\sim \lf(\frac{|z_j-z_{j+m}|}{r_{j+m}}\r)^{-\frac{2\kappa}{p}p_0+\frac{2p_0}{p}-2}\noz.
\end{align}

Moreover, from \eqref{cz kernel condition-2} and \eqref{fj proper-3},
we deduce that, for any $z\in C_1Q_j$,
\begin{align}\label{riz-upper-bdd-case-iii}
\lf|\mathcal B(f_{j+m})(z)\r|&\le\int_{Q_{j+m}}\lf|K_{\mathcal B}(z-\xi)-
K_{\mathcal B}(z-z_{j+m})\r|\lf|f_{j+m}(\xi)\r|\,d\xi\\
&\ls\int_{Q_{j+m}}\frac{r_{j+m}}{|z_j-z_{j+m}|^3}\lf|f_{j+m}(\xi)\r|\,d\xi
\ls \lf[w\lf(Q_{j+m}\r)\r]^{\frac{\kappa-1}p}\frac{r_{j+m}^3}{|z_j-z_{j+m}|^3}\noz.
\end{align}
Then, by \eqref{riz-upper-bdd-case-iii}, the fact that
$\{r_j\}_{j\in\nn}$ is non-increasing in $j$,
the  H\"older and the reverse H\"older inequalities, we conclude that
\begin{align}\label{riesz-comm-upper-bdd-on-A2Ij-ii}
&\lf\{\int_{C_1 Q_j}\lf|\lf[b(z)-\alpha_{Q_{j+m}}(b)\r]\mathcal B(f_{j+m})(z)\r|^pw(z)\,dz\r\}^{1/p}\lf[w\lf(C_1 Q_j\r)\r]^{-\kappa/p}\\
&\quad\ls \lf[w\lf(Q_{j+m}\r)\r]^{\frac{\kappa-1}p}\frac{r_{j+m}^3}{|z_j-z_{j+m}|^3}
  \lf[w\lf(C_1 Q_j\r)\r]^{-\kappa/p}\lf[\int_{C_1 Q_j}\lf|b(z)-\alpha_{Q_{j+m}}(b)\r|^p
  w(z)\,dz\r]^{1/p}\noz\\
&\quad\ls \lf[w\lf(Q_{j+m}\r)\r]^{\frac{\kappa-1}p}\frac{r_{j+m}^3}{|z_j-z_{j+m}|^3}
  \lf[w\lf(C_1 Q_j\r)\r]^{\frac{1-\kappa}p}\lf(\log\frac{|z_j-z_{j+m}|}{r_{j+m}}
+\log\frac{|z_j-z_{j+m}|}{r_{j}}\r)\noz\\
&\quad\ls \lf[w\lf(Q_{j+m}\r)\r]^{\frac{\kappa-1}p}\frac{r_{j+m}^3}{|z_j-z_{j+m}|^3}
  \lf[w\lf(\frac{|z_j-z_{j+m}|}{r_{j+m}} Q_{j+m}\r)\r]^{\frac{1-\kappa}p}
  \log\frac{|z_j-z_{j+m}|}{r_{j+m}}\noz\\
&\quad\ls \lf(\frac{|z_j-z_{j+m}|}{r_{j+m}}\r)^{-\frac{2\kappa}{p}p_0+\frac{2p_0}{p}-3}
   \log\frac{|z_j-z_{j+m}|}{r_{j+m}}\noz.
\end{align}
Notice that, for $C_1$ large enough, by \eqref{I-j-pro},
we know that $|z_j-z_{j+m}|$ is also large enough and hence
\begin{align}\label{log decrease}
\lf(\frac{|z_j-z_{j+m}|}{r_{j+m}}\r)^{-1}\log\frac{|z_j-z_{j+m}|}{r_{j+m}}\lesssim1.
\end{align}
Therefore, from \eqref{riesz-comm-upper-bdd-on-A2Ij-i},
\eqref{riesz-comm-upper-bdd-on-A2Ij-ii}, \eqref{log decrease} and $p_0\in(1,p)$,
we deduce that, for $C_1$ large enough,
\begin{align*}
&\lf\{\int_{C_1 Q_j}\lf|[b,\,\mathcal B](f_{j+m})(z)\r|^pw(z)\,dz\r\}^{1/p}
 \lf[w\lf(C_1 Q_j\r)\r]^{-\kappa/p}\\
&\quad\le\lf\{\int_{C_1 Q_j}\lf|\mathcal B\lf(\lf[b-\alpha_{Q_{j+m}}(b)\r]f_{j+m}\r)(z)\r|^p
 w(z)\,dz\r\}^{1/p}\lf[w\lf(C_1 Q_j\r)\r]^{-\kappa/p}\\
&\quad\quad+\lf\{\int_{C_1 Q_j}\lf|\lf[b(z)-\alpha_{Q_{j+m}}(b)\r]\mathcal B(f_{j+m})(z)\r|^p
 w(z)\,dz\r\}^{1/p}\lf[w\lf(C_1 Q_j\r)\r]^{-\kappa/p}\\
&\quad\lesssim\lf(\frac{|z_j-z_{j+m}|}{r_{j+m}}\r)^{-\frac{2\kappa}{p}p_0+\frac{2p_0}{p}-2}
\lf[1+\lf(\frac{|z_j-z_{j+m}|}{r_{j+m}}\r)^{-1}\log\frac{|z_j-z_{j+m}|}{r_{j+m}}\r]\\
&\quad\lesssim\lf(\frac{|z_j-z_{j+m}|}{r_{j+m}}\r)^{-\frac{2\kappa}{p}p_0+\frac{2p_0}{p}-2}
       \lesssim\lf[\frac{3C_1(r_j+r_{j+m})}{r_{j+m}}\r]^{-\frac{2\kappa}{p}p_0+\frac{2p_0}{p}-2}
       \lesssim C_1^{-\frac{2\kappa}{p}p_0+\frac{2p_0}{p}-2}
       \le \frac12 C_3 \delta C_1^{-\frac{2\kappa}{p}p_0}.
\end{align*}
This finishes the proof of \eqref{riesz-comm-upper-bdd-on-A2Ij}. By \eqref{riesz-comm-lower-bdd-on-A2Ij} and \eqref{riesz-comm-upper-bdd-on-A2Ij},
we know that, for any $j,\,m\in\nn$ and $C_1$ large enough,
\begin{align*}
&\lf\{\int_{C_1 Q_j}\lf|[b,\,\mathcal B](f_j)(z)-[b,\,\mathcal B](f_{j+m})(z)\r|^p
  w(z)\,dz\r\}^{1/p}\lf[w\lf(C_1 Q_j\r)\r]^{-\kappa/p}\\
&\quad\ge\lf\{\int_{C_1 Q_j}\lf|[b,\,\mathcal B](f_j)(z)\r|^pw(z)\,dz\r\}^{1/p}
  \lf[w\lf(C_1 Q_j\r)\r]^{-\kappa/p}\\
&\quad\quad-\lf\{\int_{C_1 Q_j}\lf|[b,\,\mathcal B](f_{j+m})(z)\r|^p
  w(z)\,dz\r\}^{1/p}\lf[w\lf(C_1 Q_j\r)\r]^{-\kappa/p}
  \ge\frac{1}{2}C_3\dz C_1^{-\frac{2\kappa}pp_0}.
\end{align*}
This finishes the proof of Lemma \ref{l-comp contra awy ori}.
\end{proof}

\begin{proof}[Proof of Theorem \ref{t-Beurling com compact}(ii)]
Without loss of generality, we may assume that $\|b\|_{{\rm BMO}(\cc)}=1$.
To show $b\in{\rm CMO}(\cc)$, noticing that $b\in {\rm BMO}(\cc)$ is a real-valued function,
we can use a contradiction argument via Lemmas \ref{l-cmo char}, \ref{l-cmo-contra}
and \ref{l-comp contra awy ori}.
Now observe that, if $b\notin {\rm CMO}(\cc)$, then $b$ does not satisfy at least one of
(i) through (iii) of Lemma \ref{l-cmo char}.
We show that $[b, \mathcal B]$ is not compact on $L_w^{p,\,\kappa}(\cc)$ in any of
the following three cases.

{\bf Case i)} $b$ does not satisfy Lemma \ref{l-cmo char}(i).
Then there exist $\dz\in(0, \fz)$ and a sequence
$$\{Q^{(1)}_j\}_{j\in\nn}:=\{Q(z_j^{(1)},r_j^{(1)})\}_{j\in\nn}$$ of squares in $\cc$
satisfying \eqref{lower bdd osci} and that $|Q^{(1)}_j|\to0$ as $j\to\fz$.
We further consider the following two subcases.

{\bf Subcase (i)} There exists a positive constant $M$ such that $|z^{(1)}_{j}|\in[0, M)$
for all $z^{(1)}_{j}$, $j\in\nn$. That is, $z^{(1)}_{j}\in Q_0:=Q(0, M)$, $\forall j\in\nn$.
Let  $\{f_j\}_{j\in\nn}$ be associated with $\{Q_j\}_{j\in\nn}$,
$\wz C_1$, $\wz C_2$, $K_0$ and $C_2$ be as in Lemmas \ref{l-cmo-contra} and
\ref{l-comp contra awy ori}.
Let $p_0\in(1,p)$ be such that $w\in A_{p_0}(\mathbb C)$ and $C_4:=3^{K_2}> C_2=3^{K_0}$
for $K_2\in\nn$ large enough such that
\begin{align}\label{C5}
C_5:=\wz C_1 C_{(w)}\delta^p 3^{2K_0(\sigma-p)}>
2\frac{\wz C_2}{1-3^{2(p_0-p)}}\frac{C_{(p_0)}}{3^{2K_2(p-p_0)}},
\end{align}
where $C_{(w)}$ is as in \eqref{e-sigma defn} and $C_{(p_0)}$ satisfies that,
for any square $Q\subset \cc$ and $t\in(1,\infty)$,
\begin{align}\label{Ap_0 double}
w(tQ)\leq C_{(p_0)} t^{2p_0}w(Q).
\end{align}
Since $|Q^{(1)}_j|\to 0$ as $j\to\fz$ and $\{z^{(1)}_j\}_{j\in\nn}\subset Q_0$,
we may choose a subsequence $\{Q_{j_\ell}^{(1)}\}_{\ell\in\nn}$ of
$\{Q^{(1)}_j\}_{j\in\nn}$ such that, for any $j\in\nn$,
\begin{equation}\label{descreasing interval}
\frac{|Q_{j_{\ell+1}}^{(1)}|}{|Q_{j_{\ell}}^{(1)}|}<\frac1{C_4^2}\quad{\rm and\quad} w\lf(Q_{j_{\ell+1}}^{(1)}\r)\le w\lf(Q_{j_{\ell}}^{(1)}\r).
\end{equation}
For fixed $\ell,\ m\in \mathbb N$, let
$$\mathcal J:=C_4Q^{(1)}_{j_\ell}\setminus C_2Q^{(1)}_{j_\ell},
\quad\mathcal J_1:=\mathcal J\setminus C_4Q^{(1)}_{j_{\ell+m}}
\quad\textrm{and}\quad\mathcal J_2:=\cc\setminus C_4Q^{(1)}_{j_{\ell+m}}.$$
Notice that
$$\mathcal J_1\subset \lf[\lf(C_4Q^{(1)}_{j_\ell}\r)\cap\mathcal J_2\r]
\quad{\rm and}\quad \mathcal J_1=\mathcal J\cap \mathcal J_2.$$
 We then have
\begin{align}\label{low lpbdd comparing com}
&\lf\{\int_{C_4Q_{j_\ell}^{(1)}}\lf|\lf[b,\mathcal B\r](f_{j_\ell})(z)
-\lf[b,\mathcal B\r](f_{j_{\ell+m}})(z)\r|^pw(z)\,dz\r\}^{1/p}\\
&\quad\ge\lf\{\int_{\mathcal J_1}\lf|\lf[b,\mathcal B\r](f_{j_\ell})(z)
-\lf[b,\mathcal B\r](f_{j_{\ell+m}})(z)\r|^pw(z)\,dz\r\}^{1/p}\noz\\
&\quad\ge\lf\{\int_{\mathcal J_1}\lf|\lf[b,\mathcal B\r](f_{j_\ell})(z)\r|^pw(z)\,dz\r\}^{1/p}
-\lf\{\int_{\mathcal J_2}\lf|\lf[b,\mathcal B\r](f_{j_{\ell+m}})(z)\r|^pw(z)\,dz\r\}^{1/p}\noz\\
&\quad=\lf\{\int_{\mathcal J\cap \mathcal J_2}
\lf|\lf[b,\mathcal B\r](f_{j_\ell})(z)\r|^pw(z)\,dz\r\}^{1/p}
-\lf\{\int_{\mathcal J_2}\lf|\lf[b,\mathcal B\r](f_{j_{\ell+m}})(z)\r|^pw(z)\,dz\r\}^{1/p}\noz\\
&\quad=:{\rm F_1}-{\rm F_2}.\noz
\end{align}

We first consider the term ${\rm F_1}$. Assume that
$E_{j_\ell}:=\mathcal J\setminus\mathcal J_2\not=\emptyset$.
Then $E_{j_\ell}\subset C_4Q^{(1)}_{j_{\ell+m}}$.
Thus, by \eqref{descreasing interval}, we have
\begin{align}\label{ee1}
|E_{j_\ell}|\le C_4^2\lf|Q^{(1)}_{j_{\ell+m}}\r|<\lf|Q^{(1)}_{j_{\ell}}\r|.
\end{align}
Now let
$$Q^{(1)}_{j_\ell,\,k}:=3^{k-1}Q^{(1)}_{j_\ell}+3^{k}r^{(1)}_{j_\ell}\vec{e},$$
where $2r^{(1)}_{j_\ell}$ is the side-length of $Q^{(1)}_{j_\ell}$.
Then, from \eqref{ee1}, we deduce that
$$\lf|Q^{(1)}_{j_\ell,\,k}\r|=3^{2(k-1)}\left|Q^{(1)}_{j_\ell}\r|> |E_{j_\ell}|. $$
By this, we further know that there exist at most two of
$\{Q^{(1)}_{j_\ell,\,k}\}_{k=K_0}^{K_2-2}$ intersecting $E_{j_\ell}$.
By \eqref{lower upper lpbdd riesz comm} and \eqref{e-sigma defn}, % and $C_1> C_2$ ,
we conclude that
\begin{align}\label{F_1^p}
{\rm F}_1^p&\ge\sum_{k=K_0,\,Q^{(1)}_{j_\ell,\,k}\cap E_{j_\ell}=\emptyset}^{K_2-2}
\int_{Q^{(1)}_{j_\ell,\,k}}\lf|[b,\mathcal B](f_{j_\ell})(z)\r|^pw(z)\,dz\\
&\ge\wz C_1\dz^p\sum_{k=K_0,\,Q^{(1)}_{j_\ell,\,k}\cap E_{j_\ell}=\emptyset}^{K_2-2}
\frac{[w(Q^{(1)}_{j_\ell})]^{\kappa -1}w(3^{k}Q^{(1)}_{j_\ell})}{3^{2kp}}\noz\\
&\ge\wz C_1 C_{(w)}\dz^p\sum_{k=K_0,\,Q^{(1)}_{j_\ell,\,k}\cap E_{j_\ell}=\emptyset}^{K_2-2}\frac{[w(Q^{(1)}_{j_\ell})]^{\kappa}}{3^{2k(p-\sigma)}}\noz\\
&\ge \wz C_1 C_{(w)}\dz^p3^{2K_0(\sigma-p)}\lf[w\lf(Q^{(1)}_{j_\ell}\r)\r]^{\kappa}
=C_5\lf[w\lf(Q^{(1)}_{j_\ell}\r)\r]^{\kappa}.\noz
\end{align}
If  $E_{j_\ell}:=\mathcal J\setminus\mathcal J_2=\emptyset$,
the inequality above still holds true.

Moreover, from \eqref{lower upper lpbdd riesz comm2}, \eqref{Ap_0 double},
\eqref{C5} and \eqref{descreasing interval}, we deduce that
\begin{align}\label{F_2^p}
\quad{\rm F}^p_2&\le\sum_{k=K_2}^{\fz}
\int_{3^{k+1}Q_{j_{\ell+m}}^{(1)} \setminus 3^{k}Q_{j_{\ell+m}}^{(1)}}
\lf|[b,\mathcal B](f_{j_{\ell+m}})(z)\r|^pw(z)\,dz\\
&\le\wz C_2\sum_{k=K_2}^\fz
\frac{[w(Q^{(1)}_{j_{\ell+m}})]^{\kappa -1}w(3^{k}Q^{(1)}_{j_{\ell+m}})}{3^{2kp}}
\le\wz C_2\sum_{k=K_2}^\fz\frac{C_{(p_0)}}{3^{2k(p-p_0)}}
  \lf[w\lf(Q^{(1)}_{j_{\ell+m}}\r)\r]^{\kappa}\noz\\
&\le \frac{\wz C_2}{1-3^{2(p_0-p)}}\frac{C_{(p_0)}}{3^{2K_2(p-p_0)}}
  \lf[w\lf(Q^{(1)}_{j_{\ell+k}}\r)\r]^{\kappa}
<\frac {C_5}2\lf[w\lf(Q^{(1)}_{j_{\ell+m}}\r)\r]^{\kappa}
 \le\frac{C_5}2\lf[w\lf(Q^{(1)}_{j_{\ell}}\r)\r]^{\kappa}.\noz
\end{align}
By \eqref{low lpbdd comparing com}, \eqref{F_1^p} and \eqref{F_2^p}, we obtain
\begin{align*}
&\lf\{\int_{C_4Q_{j_\ell}^{(1)}}\lf|[b,\mathcal B](f_{j_\ell})(z)-
  [b,\mathcal B](f_{j_{\ell+m}})(z)\r|^pw(z)\,dz\r\}^{1/p}\noz\\
&\quad\ge C_5^{1/p}\lf[w\lf(Q^{(1)}_{j_\ell}\r)\r]^{\kappa/p}-
  \lf(\frac{C_5}{2}\r)^{1/p}\lf[w\lf(Q^{(1)}_{j_\ell}\r)\r]^{\kappa/p}
  \gs \lf[w\lf(Q^{(1)}_{j_\ell}\r)\r]^{\kappa/p}.
\end{align*}
Thus, $\{[b,\,\mathcal B]f_j\}_{j\in\nn}$ is not relatively compact in $L_w^{p,\,\kappa}(\cc)$,
which implies that
$[b,\,\mathcal B]$ is not compact on $L_w^{p,\,\kappa}(\cc)$. Therefore,
$b$ satisfies condition (i) of Lemma \ref{l-cmo char}.

{\bf Subcase (ii)} There exists a subsequence
$\{Q^{(1)}_{j_\ell}\}_{\ell\in\nn}:=\{Q(z^{(1)}_{j_\ell}, r^{(1)}_{j_\ell})\}_{\ell\in\nn}$
of $\{Q^{(1)}_j\}_{j\in\nn}$ such that $|z^{(1)}_{j_\ell}|\to\infty$ as $\ell\to\infty$.
In this subcase, by $|Q_{j_\ell}^{(1)}|\to 0$ as $\ell\to\infty$,
we can take a mutually disjoint subsequence of $\{Q^{(1)}_{j_\ell}\}_{\ell\in\nn}$,
still denoted by $\{Q^{(1)}_{j_\ell}\}_{\ell\in\nn}$, satisfying \eqref{I-j-pro} as well.
This, via Lemma \ref{l-comp contra awy ori}, implies that $[b, \mathcal B]$ is not compact on $L_w^{p,\,\kappa}(\cc)$, which is a contradiction to our assumption.
Thus, $b$ satisfies condition (i) of Lemma \ref{l-cmo char}.

{\bf Case ii)} $b$ violates condition (ii) of Lemma \ref{l-cmo char}.
In this case, there exist $\delta\in(0, \fz)$ and a sequence
$\{Q^{(2)}_j\}_{j\in\nn}$ of squares in $\cc$ satisfying \eqref{lower bdd osci}
and that $|Q^{(2)}_j|\rightarrow\infty$ as $j\rightarrow\infty$.
We further consider the following two subcases as well.

{\bf Subcase (i)} There exists an infinite subsequence $\{Q^{(2)}_{j_\ell}\}_{\ell\in\nn}$
of $\{Q^{(2)}_j\}_{j\in\nn}$ and a point $z_0\in\cc$ such that, for any $\ell\in\nn$,
$z_0\in 3C_1Q^{(2)}_{j_\ell}$. Since  $|Q^{(2)}_{j_\ell}|\rightarrow\infty$ as $\ell\rightarrow\infty$, it follows that there exists a subsequence, still denoted by
$\{Q_{j_\ell}^{(2)}\}_{\ell\in\nn}$, such that, for any $\ell\in\nn$,
\begin{equation}\label{increasing interval}
\frac{|Q_{j_{\ell}}^{(2)}|}{|Q_{j_{\ell+1}}^{(2)}|}<\frac1{C_4^2}.
\end{equation}
Observe that 
$6C_1Q_{j_\ell}^{(2)}\subset 6C_1Q_{j_{\ell+1}}^{(2)}$ for any $j_\ell\in \nn$ and hence
\begin{align}\label{increasing interval weight}
w\lf(6C_1Q_{j_{\ell+1}}^{(2)}\r)\geq w\lf(6C_1Q_{j_{\ell}}^{(2)}\r)
\quad{\rm and\quad}\co\lf(b;6C_1Q_{j_\ell}\r)>\frac{\delta}{72C_1^2}.
\end{align}
We can use a similar method as that used in Subcase (i) of Case i) and
redefine our sets in a reversed order. That is, for any fixed $\ell,\,k\in\nn$, let
$$\widetilde{\mathcal J}:=6C_4C_1Q_{\ell+k}^{(2)}\setminus 6C_2C_1Q_{\ell+k}^{(2)},\quad
\widetilde{\mathcal J_1}:=\widetilde{\mathcal J}\setminus 6C_4C_1Q_{j_{\ell}}^{(2)}\quad
\textrm{and}\quad
\widetilde{\mathcal J_2}:=\cc\setminus 6C_4C_1Q_{j_{\ell}}^{(2)}.$$
As in Case i), by Lemma \ref{l-cmo-contra}, \eqref{increasing interval}
and \eqref{increasing interval weight},
we conclude that the commutator $[b, \mathcal B]$ is not compact on $L_w^{p,\,\kappa}(\cc)$.
This contradiction implies that $b$ satisfies condition (ii) of Lemma \ref{l-cmo char}.

{\bf Subcase (ii)} For any $z\in\cc$, the number of $\{3C_1Q^{(2)}_j\}_{j\in\nn}$
containing $z$ is finite. In this subcase, for each square
$Q^{(2)}_{j_0}\in \{Q^{(2)}_j\}_{j\in\nn}$,
the number of $\{3C_1Q^{(2)}_j\}_{j\in\nn}$ intersecting $3C_1Q^{(2)}_{j_0}$ is finite.
Then we take a mutually disjoint subsequence
$\{Q^{(2)}_{j_\ell}\}_{\ell\in\nn}$ satisfying
\eqref{lower bdd osci} and \eqref{I-j-pro}. From Lemma \ref{l-comp contra awy ori},
we deduce that $[b, \mathcal B]$ is not compact on $L_w^{p,\,\kappa}(\cc)$.
Thus, $b$ satisfies condition (ii) of Lemma \ref{l-cmo char}.

{\bf Case iii)} Condition (iii) of Lemma \ref{l-cmo char} does not hold true for $b$.
Then there exist $Q_0:=Q(z_0, r_0)\subset\cc$ and $\delta\in(0,\infty)$ such that,
for any $N\in\nn$ large enough,  there exists $z_N\in\cc$ such that
$|z_N|\in(N,\infty)$ and $M(b,\,Q+z_N)\in(\delta,\infty)$.
Moreover, there exists a subsequence
$\{Q^{(3)}_j\}_{j\in\nn}:=\{Q+z_{N_j}\}_{j\in\nn}$ of squares in $\cc$ such that
\begin{equation*}
\co\lf(b; Q^{(3)}_j\r)>\delta,\quad \forall j\in\nn
\end{equation*}
and
\begin{equation*}
3C_1Q^{(3)}_{\ell}\bigcap 3C_1Q^{(3)}_m=\emptyset,\quad
\forall \ell,\,m\in\nn\quad\mathrm{and}\quad\ell\neq m;
\end{equation*}
see, for example, \cite{Uchiyama78TohokuMathJ}.
Since, by  Case i) and ii), $\{Q^{(3)}_j\}_{j\in\nn}$ satisfies the conditions (i)
and (ii) of Lemma \ref{l-cmo char}, it follows that there exist positive constants
$C_{\mathrm{min}}$ and $C_{\mathrm{max}}$ such that
$$C_{\mathrm{min}}\leq r_j\leq C_{\mathrm{max}}, \quad \forall j\in\nn.$$
By this and Lemma \ref{l-comp contra awy ori}, we conclude that, if $[b, \mathcal B]$
is compact on $L_w^{p,\,\kappa}(\cc)$, then $b$  also  satisfies condition (iii) of
Lemma \ref{l-cmo char}. This finishes the proof of Theorem \ref{t-Beurling com compact}(ii)
and hence of Theorem \ref{t-Beurling com compact}.
\end{proof}

\section{An application to Beltrami equations}\label{s4}

In this section, we apply Theorem \ref{t-Beurling com compact} to show
 Theorem \ref{t-Bretrami equ}.
We use some ideas from \cite{Iwaniec92}; see also \cite{ClopCruz13AASFM}.
Recall that, for any suitable function $f$,
$$\mathcal C f(z):=\,\mathrm{p.\,v.\,} \frac1\pi\int_{\mathbb C}\frac{f(u)}{z-u}\,du,
\quad \forall z\in\cc$$
is the Cauchy transform satisfying
\begin{align}\label{Cauchy Trans}
\overline{\partial}\circ \mathcal C= Id\quad\mbox{and}
\quad\partial \circ \mathcal C=\mathcal B;
\end{align}
see \cite[p.\,112, Theorem 4.3.10]{AstalaIwaniecMartin09}.

\begin{proof}[Proof of Theorem \ref{t-Bretrami equ}]
We first prove that $(Id-b\mathcal B)^{-1}$ is bounded on $L_w^{p,\,\kappa} (\mathbb C)$.
To this end, since $Id-b\mathcal B$ is bounded on $L_w^{p,\,\kappa} (\mathbb C)$,
from a corollary of the open mapping theorem in \cite[p.\,77]{Yosida95}, we deduce that
it suffices to show that $Id-b\mathcal B$ is invertible on $L_w^{p,\,\kappa} (\mathbb C)$.

Let $P_0:=Id$ and, for any $N\in\nn$, let
$$P_N:=Id+b\mathcal B+(b\mathcal B)^2+\cdots+(b\mathcal B)^N.$$
Then we deduce that, for any $N\in\nn$,
\begin{align}\label{pde-oper multi}
(Id-b\mathcal B)P_{N-1}&=P_{N-1}(Id-b\mathcal B)
=Id-(b\mathcal B)^N\\
&=\lf[Id-b^N\mathcal B^N\r]+\lf[b^N \mathcal B^N-(b\mathcal B)^N\r]
=:\lf[Id-b^N\mathcal B^N\r]+K_N\noz.
\end{align}
Observe that, for each $N\in\nn$, $K_N$ consists of a finite summation of operators
that contain the commutator $[b, \mathcal B]$, $b$
and $\mathcal B$ as factors. Recall that, if $T$ is bounded and $S$ is compact
on a Banach space ${\mathcal X}$, then the operators
$TS$ and $ST$ are both compact on ${\mathcal X}$.
Thus, from Theorem \ref{t-Beurling com compact}(i), $\|b\|_{L^\fz(\mathbb C)}^N<1$
and the boundedness of $\mathcal B$ on $L_w^{p,\,\kappa} (\mathbb C)$
(by Theorem 3.3 in \cite{KomoriShirai09MathNachr}), we deduce that
$K_N$ is compact on $L_w^{p,\,\kappa} (\mathbb C)$. Moreover,
the $N$-th iterate $\mathcal B^N$ of $\mathcal B$
is another convolution Calder\'on-Zygmund operator with kernel
$$K_{\mathcal{B}^N}(z)=\frac{(-1)^N N}{\pi}\frac{\bar z^{N-1}}{z^{N+1}};$$
see \cite[p.\,73]{Stein70} or \cite[pp.\,101-102]{AstalaIwaniecMartin09}.
Arguing as in the proof of Lemma \ref{l-sub opr bdd}, we conclude that the operator norm
$\|{\mathcal B}^N\|_{L_w^{p,\,\kappa} (\mathbb C)\to L_w^{p,\,\kappa} (\mathbb C)}$
depends linearly on both the norm
$\|{\mathcal B}^N\|_{L_w^{p} (\mathbb C)\to L_w^{p} (\mathbb C)}$ and
the Calder\'on-Zygmund constant $$\|{\mathcal B}^N\|_{CZ}:=
\inf\lf\{C\in(0,\fz):\ C\ \mathrm{satisfies}\ \eqref{cz kernel condition-1}\ \mathrm{and}\
\eqref{cz kernel condition-2}\r\}.$$
Since both quantities are bounded by a harmlessly constant multiple of $N^2$
(see proofs of Theorem 1 in \cite{ClopCruz13AASFM} and Theorem 3.3 in \cite{KomoriShirai09MathNachr}, or \cite[p.\,127,\,\,Corollary 4.5.1]{AstalaIwaniecMartin09}),
we immediately deduce that
\begin{align*}
\lf\|b^N\mathcal B^N f\r\|_{L_w^{p,\,\kappa}
(\mathbb C)}\le \wz{C}N^2\|b\|_{L^\fz(\mathbb C)}^N\|f\|_{L_w^{p,\,\kappa} (\mathbb C)}
\end{align*}
for some positive constant $\wz{C}$ independent of $f,\ b$ and $N$.
This implies that, for large enough $N\in\nn$ such that
$$\wz{C}N^2\|b\|_{L^\fz(\mathbb C)}^N<1,$$
the operator $Id-b^N\mathcal B^N$ is invertible on $L_w^{p,\,\kappa} (\mathbb C)$.

We now deduce, from the invertibility
of $Id-b^N\mathcal B^N$   and  \eqref{pde-oper multi}, that
$$(Id-b\mathcal B)P_{N-1}\left(Id-b^N\mathcal B^N\right)^{-1}=Id+K_N\left(Id-b^N\mathcal B^N\right)^{-1}$$
and
$$\left(Id-b^N\mathcal B^N\right)^{-1}P_{N-1}(Id-b\mathcal B)=Id+\left(Id-b^N\mathcal B^N\right)^{-1}K_N.$$
This further implies that $Id-b\mathcal B$ is a Fredholm operator
 (see, for example, \cite[p.\,169]{Brezis}). Now, we apply the index theory to $Id-b\mathcal B$ as follows.
Since the continuous deformation $Id-tb\mathcal B$ for $t\in[0, 1]$ is a homotopy
from the identity operator $Id$ to $Id-b\mathcal B$, from the homotopical invariance of index,
we deduce that
$$\mathrm{Index}\lf(Id-b\mathcal B\r)=\mathrm{Index}(Id)=0.$$
Moreover, since any injective operator with index 0 is also onto, to obtain
the invertibility of $Id-b\mathcal B$, it remains to show that it is injective in
$L_w^{p,\,\kappa}(\cc)$. Assume that
$f\in L_w^{p,\,\kappa} (\mathbb C)$ satisfies that $f=b\mathcal B f$
on $L_w^{p,\,\kappa} (\mathbb C)$.
Then $f(z)=b(z)\mathcal B f(z)$ for $w$-almost every $z\in\cc$.
Moreover, the fact that $\supp(b)$ is compact implies that $f$ also has a compact support.
From this and $f\in L_w^{p,\,\kappa} (\mathbb C)$,
we further deduce that $f\in L_w^p(\cc)$.  Recall that
$Id-b \mathcal B$ is injective on $L_w^p(\cc)$ for any $p\in(1, \fz)$, see
\cite[p.\,101]{ClopCruz13AASFM}.
Thus $f=0$ in $L_w^{p} (\mathbb C)$ and hence $f(z)=0$ for $w$-almost every $z\in\cc$.
This shows that $Id-b\mathcal B$ is also injective and hence invertible on
$L_w^{p,\,\kappa} (\mathbb C)$.

As $(Id-b\mathcal B)^{-1}$ is bounded on $L_w^{p,\,\kappa} (\mathbb C)$,
we conclude that, for any $g\in L_w^{p,\,\kappa} (\mathbb C)$,
$$\|g\|_{L_w^{p,\,\kappa} (\mathbb C)}\ls \lf\|(Id-b\mathcal B)g\r\|_{L_w^{p,\,\kappa}
(\mathbb C)}.$$
Thus, for any $g\in L_w^{p,\,\kappa} (\mathbb C)\bigcap L^r(\cc)$,
let $f:=\mathcal C(Id-b \mathcal B)^{-1}g$.
By \eqref{Cauchy Trans}, we then have
$$\overline{\partial } f-b\partial f=g.$$
That is, $f$ satisfies \eqref{e-Bre equ}. Moreover,
\begin{align*}
\lf\||D f|\r\|_{L_w^{p,\,\kappa} (\mathbb C)}&\le \lf\|\overline{\partial} f\r\|_
{L_w^{p,\,\kappa} (\mathbb C)}+\lf\|\partial f\r\|_{L_w^{p,\,\kappa} (\mathbb C)}\\
&\ls\lf\|(Id-b\mathcal B)^{-1}g\r\|_{L_w^{p,\,\kappa} (\mathbb C)}+
\lf\|\mathcal B(Id-b\mathcal B)^{-1}g\r\|_{L_w^{p,\,\kappa} (\mathbb C)}
\ls\|g\|_{L_w^{p,\,\kappa} (\mathbb C)}.
\end{align*}

For the uniqueness, choosing two solutions $f_1$ and $f_2$ of \eqref{e-Bre equ},
the difference $f_0:=f_1-f_2$ satisfies that
$\overline{\partial}f_0-b\partial f_0=0$ and $|Df_0|\in L^r(\cc)$,
which implies that $(Id-b\mathcal B)(\overline{\partial}f_0)=0$
because $\mathcal B\circ \overline{\partial}=\partial$
(see, for example, \cite[p.\,162]{AstalaIwaniecMartin09}).
From \cite[p.\,43]{Iwaniec92} and $b\in \mathrm{CMO}(\cc)$,
we deduce that $Id-b\mathcal B$ is injective in $L^r(\cc)$.
Thus, $\overline{\partial}f_0=0$ and therefore
$\partial f_0=\mathcal B (\overline{\partial}f_0)=0$.
Accordingly, $|Df_0|=0$ and hence $f_0$ is a constant.
This finishes the proof of Theorem \ref{t-Bretrami equ}.
\end{proof}

\bigskip

\noindent Jin Tao and Dachun Yang

\medskip

\noindent Laboratory of Mathematics and Complex Systems
(Ministry of Education of China),
School of Mathematical Sciences, Beijing Normal University,
Beijing 100875, People's Republic of China

\smallskip

\noindent {\it E-mails}: \texttt{jintao@mail.bnu.edu.cn} (J. Tao)

\noindent\phantom{{\it E-mails:} }\texttt{dcyang@bnu.edu.cn} (D. Yang)

\bigskip

\noindent Dongyong Yang (Corresponding author)

\medskip

\noindent School of Mathematical Sciences, Xiamen University, Xiamen 361005,  China

\smallskip

\noindent {\it E-mail}: \texttt{dyyang@xmu.edu.cn }

\end{document}